\RequirePackage{fix-cm}
\documentclass[smallextended,numbook]{svjour3}       \smartqed  \usepackage{graphicx}
\usepackage{amsmath}
\usepackage{amssymb}
\usepackage{hyperref}
\usepackage{cleveref}
\usepackage{setspace}
\usepackage{todonotes}
\usepackage{nicefrac}
\usepackage{scrextend}
\usepackage{graphicx}
\usepackage{tikz}
\usepackage{geometry}
\usepackage{csquotes}
\usepackage[normalem]{ulem}
\usepackage{nicefrac}
\usepackage{bbm}
\usepackage{siunitx}
\usepackage{todonotes}
\usepackage[sort]{cite}
\usepackage{atbegshi}
\usepackage{refcount}

\begin{document}

\newenvironment{delayedproof}[1]
 {\begin{proof}[\raisedtarget{#1}Proof of \Cref{#1}]}
 {\end{proof}}
\newcommand{\raisedtarget}[1]{  \raisebox{\fontcharht\font`P}[0pt][0pt]{\hypertarget{#1}{}}}
\newcommand{\proofref}[1]{\hyperlink{#1}{proof}}
\newtheorem{scheme}{Scheme}  
\newtheorem{assumption}{Assumption}  
\def\Nn{{\mathbb N}}
\def\Zz{{\mathbb Z}}
\def\Rr{{\mathbb R}}
\def\Tt{{\mathbb T}}
\def\Cc{{\mathbb C}}
\def\Qq{{\mathbb Q}}
\def\Ff{{\mathbb F}}
\def\Kk{{\mathbb K}}
\def\Hh{{\mathbb H}}
\def\Ee{{\mathbb E}}
\def\Ii{{\mathbb I}}
\def\id{{\rm id}}
\def\dif{{\rm d}}
\def\dim{{2}}
\newcommand{\ubp}{{\bf\bar{u}}^{n+1}}
\newcommand{\ubo}{{\bf\bar{u}}^{n}}
\newcommand{\ubm}{{\bf\bar{u}}^{n-1}}
\newcommand{\uhp}{{\bf\hat{u}}^{n+1}}
\newcommand{\uho}{{\bf\hat{u}}^{n}}
\newcommand{\uhm}{{\bf\hat{u}}^{n-1}}
\newcommand{\pp}{{p^{n+1}}}
\newcommand{\po}{{p^{n}}}
\newcommand{\pmm}{{p^{n-1}}}
\newcommand{\bff}{{\bf{}f}}
\newcommand{\tD}{{\textmd{D}}}
\newcommand{\inner}[1]{{\langle#1\rangle}}
\newcommand{\norm}[1]{{\lVert#1\rVert}}
\newcommand{\normb}[1]{{\left\lVert#1\right\rVert}}
\newcommand{\abs}[1]{{\left|#1\right|}}
\newcommand{\commentamw}[1]{{\color{purple}(A: #1)}}
\newcommand{\commentbw}[1]{{\color{teal}(B: #1)}}
\newcommand{\commentjz}[1]{{\color{orange}(J: #1)}}
\newcommand{\bu}{{\bar{\bf u}}}
\newcommand{\be}{{\bar{\bf e}}}
\newcommand{\hu}{{{\bf u}}}
\newcommand{\he}{{{\bf e}}}
\newcommand{\hf}{{{\bf f}}}
\newcommand{\hh}{{{\bf h}}}
\newcommand{\bfu}{{\bf u}}
\newcommand{\bfv}{{\bf v}}
\newcommand{\mcE}{{\mathcal{E}}}
\newcommand{\ubar}[1]{\underaccent{\bar}{#1}}
\setlength{\parindent}{0pt} 

\title{The generalized scalar auxiliary variable applied to the incompressible Boussinesq Equation}

\author{Andreas Wagner\thanks{Corresponding author: wagneran@cit.tum.de} \and Barbara Wohlmuth \and Jan Zawallich}

\institute{
        Andreas Wagner \and Barbara Wohlmuth \and Jan Zawallich \at
        TUM School of Computation, Information and Technology,
        Technical University of Munich, Bavaria, Germany \\
        \email{wagneran@cit.tum.de, wohlmuth@cit.tum.de, zaw@cit.tum.de}
}

\date{Received: date / Accepted: date}

\maketitle

\begin{abstract}
    This paper introduces a second-order time discretization for solving the incompressible Boussinesq equation.
It uses the generalized scalar auxiliary variable (GSAV) and a backward differentiation formula (BDF), based on a Taylor expansion around $t^{n+k}$ for $k\geq3$. 
An exponential time integrator is used for the auxiliary variable to ensure stability independent of the time step size.
We give rigorous asymptotic error estimates of the time-stepping scheme, thereby justifying its accuracy and stability.
The scheme is reformulated into one amenable to a $H^1$-conforming finite element discretization. 
Finally, we validate our theoretical results with numerical experiments using a Taylor--Hood-based finite element discretization and show its applicability to large-scale 3-dimensional problems.
\end{abstract}

\section{Introduction}\label{sec:introduction}
The Boussinesq system~\cite{zeytounian2003joseph} for an incompressible fluid models the temperature evolution when gravity and buoyancy are the dominant forces on the fluid. For engineering applications, it is a vital component to describe cooling, ventilation, or the dispersion of heat~\cite{de1983natural,ward2000application,BAKER1994261,lube2008stabilized}.
Aside from that, it is used in various scientific fields: 
It is integral for geophysical applications like earth-mantle convection, though here, the infinite Prandtl number approximation leads to an omission of the inertia terms in the fluid equation~\cite{gassmoller2020formulations}. Further, it is also applied to describe pollutant dispersion in environmental engineering~\cite{MADALOZZO20145883} or to model the convection in stars~\cite{spiegel1971convection} in astrophysics.
All in all, the applications are various, and hence, finding efficient discretizations for the Boussinesq system is of interest.

Formally, the system can be stated in a non-dimensionalized form as follows: For the temperature $\theta\colon[0,T]\times\Omega\to\Rr$, the
velocity $\hu\colon [0,T]\times\Omega  \to \Rr^d$,
and the pressure $p\colon [0,T]\times\Omega  \to \Rr$ 
on the bounded domain $\Omega \subset \Rr^d$ with time interval $[0,T]$, the incompressible Boussinesq system~\cite{zeytounian2003joseph} is given by
\begin{align}
	\partial_t \theta 
	+ (\hu \cdot \nabla) \theta
	- \frac{1}{\textmd{Re Pr}}\Delta \theta
	=                                                & \, 0
	                                                 & \textmd{ in } [0,T]\times\Omega,
    \label{eq:boussinesq:heat}
	\\
	\partial_t \hu
	+
	(\hu \cdot \nabla) \hu
	- \frac{1}{\textmd{Re}} \Delta \hu + \nabla p  = & \, \textmd{Ri } \theta {\bf e}_g
	                                                 & \textmd{ in } [0,T]\times\Omega,
    \label{eq:boussinesq:momentum}
	\\
	\nabla \cdot \hu =                               & \, 0
	                                                 & \textmd{ in } [0,T]\times\Omega,
    \label{eq:boussinesq:mass}
\end{align}
with the Reynolds number $\textmd{Re}$, Prandtl number $\textmd{Pr}$, and Richardson number $\textmd{Ri}$ as material parameters, and where ${\bf e}_g$ is the vector describing gravity. We use the boundary conditions $\hu = 0$, and $\partial_n \theta = 0$ on $[0, T]\times\partial \Omega$, and the initial conditions $\hu(0,\cdot) = \hu^0$, $\theta(0,\cdot)=\theta^0$, for the purely spatial fields $\theta^0$, and $\hu^0$.
The model generally neglects temperature-induced density changes and only incorporates buoyancy effects due to gravity.

Over the years, several time-stepping schemes have been proposed:
In \cite{davis2002operator}, a multiple-step time integrator is used for the fluid equation, while a backward Euler and a Crank-Nicolson scheme are applied to the energy equation. Both, fluid and energy equations use semi-implicit advective parts.
A first-order convergent consistent projection scheme was proposed in~\cite{yang2023error}, including stability proofs and an error analysis incorporating the spatial discretization.
In~\cite{jiang2019pressure}, a similar analysis is carried out for an ensemble scheme, including a particle mean velocity.
Moreover, additional first and second-order IMEX schemes with stability proofs have been introduced in~\cite{chen2023unconditional}.
A fully implicit scheme, including a nonlinear approximation, strategy has been proposed in~\cite{damanik2009monolithic}.
In~\cite{ding2024optimal}, a second-order scheme including temperature-dependent viscosity and heat conduction coefficient has been introduced.
In~\cite{takhirov2021direction}, the Boussinesq Equation on the spherical shell is solved with a direction splitting scheme using a domain-decomposition approach. 
Similarly, \cite{paszynski2020massively} uses a splitting scheme and presents a massively parallel simulation for atmospheric flows.
A mostly explicit method with an implicitly treated pressure term is presented in a high-performance context in~\cite{castillo2024explicit} with an analysis for the fluid-dynamic parts in~\cite{kaya2024error}.
In~\cite{pan2022monolithic}, a staggered scheme for the momentum and energy equations is used, also considering additional temperature-related changes in the density and viscosity. 
A scheme with proofs for nonsmooth data was presented in~\cite{zhang2018decoupled}.
Further, various schemes have been introduced when just the steady state is of interest~\cite{hou2022decoupled, hawkins2024analysis, pollock2021acceleration}.

More recently, a scalar auxiliary variable (SAV) approach~\cite{shen2018scalar,shen2018convergence} 
and more specifically, its extension to multiple scalar auxiliary variables (MSAV)~\cite{cheng2018multiple}, was applied to a Crank-Nicolson leap-frog scheme and a BDF2 scheme in \cite{jiang2023unconditionally} in combination with stability terms for the spatial discretization.
The unconditional stability was proven, and the second-order convergence was shown numerically.
Though the scheme succeeds in linearizing the Boussinesq system, it still requires inverting two saddle point problems for the velocity and pressure, which typically requires advanced iterative linear solvers.
In~\cite{zhang2024error}, a first-order SAV scheme based on a projection method is derived, including a stability and error analysis with remarks on how to generalize it to a second-order scheme.

Recently, building on the SAV approach, a generalized scalar auxiliary variable (GSAV) approach \cite{huang2022new} has been introduced for dissipative systems and successfully applied to the Navier--Stokes equation, obtaining a 
first-order in time consistent splitting scheme~\cite{li2023error} and a second-order scheme~\cite{huang2023stability} based on a BDF scheme using a time-discretization at $t^{n+5}$.
These schemes require the inversion of two generalized Poisson equations per time step for the velocity and pressure, which is significantly easier to achieve, and computationally less expensive.
Hence, in this paper, we build on the scheme introduced in~\cite{huang2023stability} to achieve a more efficient discretization for the Boussinesq equation.
We thereby introduce a new auxiliary variable for stability and include temperature dependence in our error analysis.
In addition, we justify the approach's applicability for general Taylor expansions around $t^{n+k}$ for $k\geq3$, instead of limiting it to $k=5$.
In the following, we want to present our time discretization for the Boussinesq equation, heavily relying on the analysis and time-stepping scheme of the incompressible Navier--Stokes in~\cite{huang2023stability}.
We extend this scheme in Sec.~\ref{sec:scheme} to the incompressible Boussinesq equation by including a bidirectional coupling to a temperature discretization.
We will introduce two equivalent schemes, one having higher regularity requirements for the discretization inspired by \cite{huang2023stability}, while for the other scheme, the ideas of \cite{shen2007error} are used to rewrite our scheme in a feasible way for a finite element discretization.
Our main theoretical result, an a priori error estimate for the time-discrete algorithm for $d=2$, is introduced in Sec.~\ref{sec:assumption-and-main-result}, with a rigorous proof in Sec.~\ref{sec:theory}.
Finally, in Sec.~\ref{sec:numerical}, we show numerical experiments validating our theoretical results, as well as 2D and 3D results for the Marsigli flow.
\section{Time discretization and weak formulation}\label{sec:scheme}
We will consider a slightly more general problem than the instationary Boussinesq system. This problem is given by
\begin{align}
    \partial_t \theta 
    + (\hu \cdot \nabla) \theta 
    - \kappa \Delta \theta 
    =                              & \, g
                                   & \textmd{ in } [0,T]\times\Omega,
    \label{eq:gen:heat}
    \\
    \partial_t \hu
    +
    (\hu \cdot \nabla) \hu
    - \nu \Delta \hu + \nabla p  = & \, \hf(\theta)
                                   & \textmd{ in } [0,T]\times\Omega,
    \label{eq:gen:momentum}
\end{align}
with constraint $ \nabla \cdot \hu =\, 0 \textmd{ in } [0,T]\times\Omega$,
and boundary conditions
$\theta = 0, \textmd{ and } \hu = 0 \textmd{ on } [0,T]\times\partial \Omega$, with the right-hand side $\hf(t, \theta(t)) = \hf_1(t) + \hf_2(\theta(t))$.
We assume that $\hf_1$ is uniformly bounded in time by the constant $C_{\hf_1}$, and $\hf_2$ Lipschitz continuous, i.e.~$|\hf_2(\theta) - \hf_2(\tilde \theta)| \leq \alpha |\theta - \tilde \theta|$, and $\hf_2(0)=\mathbf{0}$.
Furthermore, g is uniformly bounded in time by $C_g$.
Also, we assume appropriate initial conditions at $t=0$ for $\theta$, $\hu$, and $p$.

In the following derivation, we will assume that all fields are sufficiently smooth and later specify more restrictive assumptions. Our main goal is to decouple temperature, velocity, and pressure into independent equations.
First from $\nabla \cdot \hu = 0$ and the homogeneous Dirichlet boundary conditions, we have $\inner{\partial_t \hu, \nabla q} = 0$, for all $q \in H^1(\Omega)$, by differentiating the divergence term in $t$, testing with $q \in H^1(\Omega)$, applying partial integration and the homogeneous boundary conditions for $\bfu$. Since $\nu \nabla \nabla \cdot \hu = 0$, for $\hu$ regular enough, we can use Eq.~\eqref{eq:gen:momentum} and the identity
$\Delta - \nabla\nabla\cdot = - \nabla \times \nabla \times$, to obtain for the pressure
\begin{align}
\inner{
\nabla p, \nabla q
}
&= \inner{\bff(\theta) - (\hu \cdot \nabla)\hu - \nu \nabla \times \nabla \times \hu, \nabla q }
,
\label{eq:gen:pressure}
\end{align}
where $\inner{u,v} = \int_\Omega u \cdot v \,\textmd{d}x$ denotes the inner product in the space of square-integrable functions.
In our time discretization, this equation replaces the incompressibility constraint and allows us to decouple pressure and velocity, yielding a \emph{consistent splitting scheme} \cite{guermond2006overview}.
Further, we will use the GSAV method to stabilize our system. For this, we define the energy
\begin{align*}
    \mathcal{E}(\theta(t), \hu(t)) & = 
    \frac{1}{2} \norm{\hu(t)}^2
    + \frac{\bar\alpha^2}{2} \norm{\theta(t)}^2
    ,
\end{align*}
where 
$\norm{\hu(t)}^2 = \int_\Omega |\hu(t,x)|^2 \textmd{d}x$
and
$\norm{\theta(t)}^2 = \int_\Omega |\theta(t,x)|^2 \textmd{d}x$
denote the standard norm for square-integrable functions,
and $\bar\alpha > 0$ is some constant chosen later, {linearly} depending on the Lipschitz constant $\alpha$ of $\bff_2$.
The energy evolution follows the dissipation law
\[
    \frac{d \mathcal{E}}{dt} 
    =
    - \nu \norm{\nabla \hu}^2 + \inner{\hf(\theta), \hu}
    - \kappa \bar\alpha^2 \norm{\nabla \theta}^2 + \bar\alpha^2 \inner{g, \theta}
    ,
\]
where we plugged in Eqs.~\eqref{eq:gen:heat}, and \eqref{eq:gen:momentum} for the time derivative, utilized $\inner{(\bfu \cdot \nabla) \theta, \theta} = 0$, and $\inner{(\bfu \cdot \nabla) \bfu, \bfu} = 0$ due to the antisymmetry of the advection terms.
{W}e define the auxiliary variable as
\[
    r =  \mathcal{E}(\theta, \hu) + \bar{C}
    =
    \frac{1}{2} \norm{\hu}^2
    +
    \frac{\bar\alpha^2}{2}  \norm{\theta}^2
    + \bar{C}
    ,
\]
where $\bar{C}$ is sufficiently large.
The idea is now to solve the auxiliary variable with a separate ordinary differential equation (ODE). We save the {ratio} between the energies from the ODE and partial differential equation (PDE) in the variable $\xi$. Scaling our velocity field by $\eta = 1 - (1-\xi)^2$ means that we dampen our velocity if there is a mismatch between both energies, allowing us to obtain stability.
To derive an ODE for $r$, we start with $\dot{r} = \frac{d\mathcal{E}}{dt}(\theta, \hu)$, such that $r = \mathcal{E}(\theta, \hu) + \bar{C}$.
Just multiplying the right-hand side with one gives the differential equation 
\(
\dot{r} = \frac{r}{\mathcal{E}(\theta, \hu) + \bar{C}} \frac{d\mathcal{E}}{dt}(\theta, \hu)
,
\)
which has the solution
\begin{align}
r(t^{n+1}) = \exp\left(\int_{t^n}^{t^{n+1}}\frac{\nicefrac{d\mathcal{E}}{dt}(\theta, \hu)}{\mathcal{E}(\theta, \hu) + \bar{C}}\,\textmd{d}t \right) r(t^{n})
.
\label{eq:r-time-integration}
\end{align}
This reformulation has the advantage that finding a bound for the exponential term is straightforward. In practice, we will apply a simplistic quadrature formula to approximate the integral.

Next, we want to combine the previous result into a time discretization.
For a given $N\in\mathbb{N}$, we first discretize the interval $[0,T]$ equidistantly by $t^n=n\tau$ with $\tau=\frac{T}{N}$ for $n\in\Ii := [0,N]\cap\Nn_0$.
We now want to find a time-discrete solution sequence $(\hu^n,p^n,\theta^n)_{n \in \Ii}$, approximating the real solution $(\hu^n_\star, p^n_\star, \theta^n_\star)_{n \in \Ii}$,
where the star denotes the discretization of the time continuous solution, i.e., $\hu_\star^n = \hu(t^n)$, $p_\star^n = p(t^n)$, and $\theta_\star^n = \theta(t^n)$.
We introduce a time-discrete numerical scheme based on \cite{huang2023stability}.
We use the second-order accurate time derivative $\nicefrac{1}{2\tau} \tD^k$ at time point $t^{n+k}$, and the extrapolation operator $\delta^k$, which are defined for a sequence
$(v^{n})_{n\in\mathbb{I}}$
via
\begin{align}
    \tD^k\! v^{n+1} :=
    (2k+1)v^{n+1} - 4kv^{n} + (2k-1)v^{n-1},
    \quad\textmd{ and }\quad
    \delta^{k} v^{n} := k v^{n} - (k-1) v^{n-1}
    ,
    \label{eq:def-D-delta}
\end{align}
for fixed real $k \geq 1$ and an integer $n \geq 1$, to derive a second-order accurate scheme.
We use $\nicefrac{1}{2\tau}\tD^l$, and $\nicefrac{1}{2\tau}\tD^k$ as time derivatives {for the temperature and velocity at possibly different points $t^{n+l}$ and $t^{n+k}$}, linearize Eqs.~\eqref{eq:gen:heat}, \eqref{eq:gen:momentum}, and \eqref{eq:gen:pressure} by extrapolating the fields in nonlinearities, 
introduce a scaled $\hu$, and an unscaled $\bu$ velocity variable for stability, and obtain:
\begin{scheme}\label{scheme:incompr-instat-boussinesq}
    Given a second order in time accurate approximation $(\theta^1, \hu^1, p^1)$, we have for $1 \leq n \leq N$ the scheme
    \begin{align}
         &
        \tD^l \theta^{n+1}
        + 2 \tau (\delta^{{l}+1} \hu^{n} \cdot \nabla)\delta^{l+1} \theta^{n}
        - 2 \tau \kappa \Delta \delta^{l} \theta^{n+1}
        = 2 \tau g(t^{n+l}),
        \label{eq:scheme-T}
        \\
         &
        \tD^k \ubp
        + 2\tau (\delta^{k+1} \hu^n \cdot \nabla) \delta^{k+1}\hu^{n}
        - 2\tau \nu \Delta \delta^k \bu^{n+1}
        + 2\tau \nabla \delta^{k+1} p^n =
        2\tau \hf(t^{n+k}, \delta^{k+1}\theta^{n})
        ,
        \label{eq:scheme-u}
        \\
         &
        \inner{\nabla\pp, \nabla q} = \inner{\hf(t^{n+1}, \theta^{n+1}) - \ubp \cdot \nabla \ubp - \nu \nabla \times \nabla \times \ubp, \nabla q}
        ,
        \label{eq:scheme-p}
        \\
         &
                \hu^{n+1} = \eta^{n+1} \bu^{n+1}
        \label{eq:scheme-scale}
        ,
    \end{align}
    where  $\eta^{n+1}$ is a scaling factor depending on $\theta^{n+1}$ and  $\bu^{n+1}$, and $l \geq 1$, $k\geq 3$ are extrapolation widths. 
    The translated energy {$r^{n+1}$ at $t^{n+1}$} with initial condition $r^0 = \nicefrac{1}{2} \norm{\hu(0)}^2 + \nicefrac{\bar{\alpha}^2}{2} \norm{\theta(0)}^2 + \bar{C}$ is described by the first-order time integrator
    \begin{align}
r^{n+1} = \exp\left(\tau\frac{\nicefrac{d\mathcal{E}}{dt}(\theta^{n+1}, \bu^{n+1})}{\mathcal{E}(\theta^{n+1}, \bu^{n+1}) + \bar{C}} \right) r^{n}
,
        \label{eq:scheme-r}
    \end{align}
    where we approximated the integral of Eq.~\eqref{eq:r-time-integration} by evaluating the integrand at the rightmost point.
    Further, the scaling variable $\eta^{n+1}$ is described by
    \begin{align}
        \xi^{n+1}  = & \, \frac{r^{n+1}}{\mathcal{E}(\theta^{n+1}, \bar{\bf u}^{n+1}) + \bar{C}},
        \quad
        \eta^{n+1} = 1 - (1 - \xi^{n+1})^2
        .
        \label{eq:scheme-eta-xi}
    \end{align}
\end{scheme}

\begin{remark}
    The scheme augments the one in \cite{huang2023stability} by including the heat equation and introducing a bidirectional coupling via an advection term in the heat equation and a nonlinear right-hand side term in the momentum equation.
    Note that since $\theta$, $\hu$, and $p$ are decoupled, approximating $\theta^{n+1}$, $\hu^{n+1}$ and $p^{n+1}$ just requires {inverting shifted Laplacians.}
    Moreover, our time integrator automatically guarantees positivity.
    The constant $\bar{C}$ only depends on the size of ${\bf f}_1$ and $g$, and not on the time dependent ${\bf f}_2(\theta)$ or the size of $\tau$ (see Lemma~\ref{lem:weak-stability}).
\end{remark}

For $H^1$-conforming finite element functions, the term $\inner{\nabla \times \nabla \times \hu^{n+1}, \nabla q}$ is not well-defined.
One option for circumventing this issue is an approach similar to Eqs.~(2.10)-(2.12) in \cite{guermond2003new} with an in-depth analysis in~\cite{shen2007error}.
To achieve this, we apply $\delta^k$ to Eq.~\eqref{eq:scheme-p} and obtain
\begin{align*}
    \inner{\nabla \delta^k p^{n+1}, \nabla q} =
    \inner{\delta^k(\hf(t^{n+1},\theta^{n+1})) 
        -\delta^k(\bu^{n+1}\cdot\nabla \bu^{n+1})
        - \nu \nabla \times \nabla \times \delta^k\bu^{n+1}, \nabla q}
    .
\end{align*}
We can rewrite the last term by using
the identity $-\nabla\times\nabla\times = \Delta - \nabla (\nabla\cdot)$ and using Eq.~\eqref{eq:scheme-u} to eliminate the newly introduced $\Delta$ terms:
\begin{align*}
    \inner{\nabla \delta^k p^{n+1}, \nabla q} = &\,
    \inner{\delta^k(\hf(t^{n+1},\theta^{n+1})) - \hf(t^{n+k},\delta^{k+1}\theta^{n}), \nabla q}
    \\&
    + \inner{(\delta^{k+1}\hu^{n}\cdot\nabla \delta^{k+1}\hu^{n}) - \delta^k(\bu^{n+1}\cdot\nabla\bu^{n+1}), \nabla q}
    \\&
    + \inner{\nicefrac{1}{2\tau}\tD^k\ubp + \nabla \delta^{k+1} p^n, \nabla q}
    - \nu \inner{ \nabla \nabla \cdot \delta^k \bu^{n+1}, \nabla q}
    .
\end{align*}
Omitting the {high}-order terms in $\tau$ simplifies the newly derived pressure equation to
\begin{align}
    \inner{\nabla \delta^k p^{n+1}, \nabla q} = &
    \inner{ \nicefrac{1}{2\tau}\tD^k\ubp + \nabla \delta^{k+1} p^n
        - \nu \nabla \nabla \cdot \delta^k \bu^{n+1}, \nabla q}
    .
    \label{eq:newpressure}
\end{align}
Hence, we obtain the following {time-discrete} scheme:
\begin{scheme}\label{scheme:fem}
    Given a second order in time accurate approximation $(\theta^1, \hu^1, p^1)$, we have for $1 \leq n \leq N$ the scheme
    \begin{align}
         &
        \tD^l \theta^{n+1}
        + 2 \tau (\delta^{l+1} \hu^{n} \cdot \nabla)\delta^{l+1} \theta^{n}
        - 2 \tau \kappa \Delta \delta^{l} \theta^{n+1}
        = 2 \tau g(t^{n+l}),
        \nonumber
        \\
         &
        \tD^k \ubp
        + 2\tau (\delta^{k+1} \hu^n \cdot \nabla) \delta^{k+1}\hu^{n}
        - 2\tau \nu \Delta \delta^k \bu^{n+1}
        + 2\tau \nabla \delta^{k+1} p^n =
        2\tau \hf(t^{n+k}, \delta^{k+1}\theta^{n})
        ,
        \nonumber
        \intertext{for temperature and velocity. We introduce the correction $\psi \in H^1(\Omega)\cap L^2_0(\Omega)$}
         & \inner{\nabla \psi^{n+1}, \nabla q}  = \nicefrac{1}{2\tau}\inner{\tD^k\ubp ,\nabla q}
        ,
        \label{eq:psi-definition}
        \intertext{to update the pressure via}
         &
        \pp = \frac{k-1}{k}\po - \nu \nabla \cdot \left( \ubp - \frac{k-1}{k} \ubo \right) + \frac{1}{k} \delta^{k+1} p^n + \frac{1}{k} \psi^{n+1}
        \label{eq:p-rec-definition}
        .
    \end{align}
    We point out that \eqref{eq:p-rec-definition} follows from \eqref{eq:newpressure} and the definition in \eqref{eq:def-D-delta} of $\delta^{k}$ by recursion.
    Again, we scale via $\hu^{n+1} = \eta^{n+1} \bu^{n+1}$, where $\eta^{n+1}$ is defined by Eq.~\eqref{eq:scheme-eta-xi}.
\end{scheme}
An a priori error analysis for Scheme~\ref{scheme:fem} will be given in Sec.~\ref{sec:theory}.
For completeness, we want to give the full spatial-temporal discretization of our algorithm in weak form, as we implemented it for our numerical experiments in Sec.~\ref{sec:numerical}.
We aim at approximating the sequence $(\theta^n_h, u^n_h, p^n_h) \in X_h \times V_h \times Q_h$, where $X_h \subset H^1_0(\Omega)$, $V_h\subset (H^1_0(\Omega))^d$, and $Q_h \subset H^1(\Omega) \cap L^2_0(\Omega)$, are $H^1$-conforming finite element spaces for the discrete temperature $\theta^n_h$, velocity $u^n_h$, and pressure $p^n_h$, which will be specified in Sec.~\ref{sec:numerical}.
We denote by $\chi_h \in X_h$, $v_h \in V_h$, and $q_h \in Q_h$ arbitrary test functions in our weak formulations.
This enables us to formulate the fully discrete scheme: 
\begin{scheme}\label{scheme:fem-spatially-discrete}
    Given initial conditions $(\theta^0_h, \hu^0_h, p^0_h)$, and a second order in time accurate approximation $(\theta^1_h, \hu^1_h, p^1_h)$,
    we setup our algorithm 
    $\bu^i_h = \hu^i_h$, 
    $\eta^i=1$,
    $r^i = \mathcal{E}(\theta^i_h, \bu^i_h) + \bar{C} $
    for the initial steps $i=0,1$.
    We note that in practice, $(\theta^1_h, \hu^1_h, p^1_h)$ are either computed by a single-step method or by interpolation in case of synthetic solutions where additional information about the solution is available.
    For $1 \leq n \leq N$, we sequentially solve the equations
    \begin{align}
         &
         \inner{
        \tD^l \theta^{n+1}_h, \chi_h
        }
        + 2 \tau 
        \inner{
        (\delta^{l+1} \hu^{n}_h \cdot \nabla)\delta^{l+1} \theta^{n}_h,
        \chi_h
        }
        + 2 \tau \kappa \inner{\nabla \delta^{l} \theta^{n+1}_h, \nabla \chi_h}
        = 2 \tau \inner{g(t^{n+l}), \chi_h},
        \nonumber
        \\
         &
        \inner{
        \tD^k \ubp_h,
        v_h}
        + 2\tau \inner{(\delta^{k+1} \hu^n_h \cdot \nabla) \delta^{k+1}\hu^{n}_h, v_h}
        + 2\tau \nu \inner{\nabla \delta^k \bu^{n+1}_h, \nabla v_h} 
        \nonumber \\& \qquad 
        + 2\tau \inner{\nabla \delta^{k+1} p^n_h, v_h} =
        2\tau \inner{\hf(t^{n+k}, \delta^{k+1}\theta^{n}_h), v_h}
        ,
        \nonumber
        \intertext{for the temperature $\theta^{n+1}_h$ and velocity $\bu^{n+1}_h$ by inverting shifted Laplacians. We solve for the correction $\psi_h \in Q_h$ and the divergence projection $s_h^{n+1} \in Q_h$ by inverting a Neumann-Laplacian and mass matrix and executing an $L^2$-projection in 
        }
         & \inner{\nabla \psi^{n+1}_h, \nabla q_h}  = \nicefrac{1}{2\tau}\inner{\tD^k\ubp_h ,\nabla q_h},
         \textmd{ and }
         \inner{s^{n+1}_h, q_h}  = \inner{\nabla \cdot (\bu^{n+1}_h - \nicefrac{k-1}{k} \bu^{n}_h), q_h}
        ,
        \nonumber
        \intertext{and update the pressure via}
         &
        p^{n+1}_h = \nicefrac{k-1}{k}p^n_h - \nu s^{n+1}_h + \nicefrac{1}{k} \delta^{k+1} p^n_h + \nicefrac{1}{k} \psi^{n+1}_h
        .
        \nonumber
    \end{align}
    Again, we scale via $\hu^{n+1} = \eta^{n+1} \bu^{n+1}$, where $\eta^{n+1}$ is defined by Eq.~\eqref{eq:scheme-eta-xi}.
\end{scheme}

After fully stating our algorithm, we will fix our notation in the next section, state our assumptions, and formulate our theoretical main result.
\section{Assumptions and main result}\label{sec:assumption-and-main-result}
In the following, we are in a standard functional analytic setting in 2-dimensions:
We denote the space of measureable functions with finite $p$-th moment by $L^p(\Omega)$ with the norm $\norm{\cdot}_{L^p(\Omega)}$.
As mentioned in the previous section, we abbreviate the norm for $L^2(\Omega)$ by $\norm{\cdot}$ and denote its inner product by $\inner{\cdot, \cdot}$.
Further, we write $L^2_0(\Omega)$ for all square-integrable functions with zero average, which is the natural space where the pressure resides.
The function space of $k$-times weakly differentiable functions with square integrable derivatives is denoted by $H^k(\Omega)$, and its norm by $\norm{\cdot}_k$.
With $H^k_0(\Omega)$ we denote the functions in $H^k(\Omega)$ with zero trace on the boundary $\partial\Omega$.
Similarly, the function spaces for vectorial fields are given by $(L^2(\Omega))^\dim$, $(H^k(\Omega))^\dim$, and $(H^k_0(\Omega))^\dim$
with norms also denoted by $\norm{\cdot}$, and $\norm{\cdot}_k$.
For a Hilbert space $V$, $C((0,T); V)$ stands for the continuous time-dependent functions on the interval $[0, T]$ with values in $V${, and 
analogously with $L^2((0,T); V)$ the square-integrable functions in time.
For $v \in C((0,T); V)$, pointwise evaluation in time is well-defined. Hence, we denote the discrete sequence matching our time-discretization by $v^n_\star = v(t^n)$.
}

For our estimates to be well-defined, velocity and temperature must be in $H^2(\Omega)$. This restricts the feasible domains $\Omega$:
\begin{assumption}[Domain $\Omega$]\label{ass:domain}
In the following, we always assume that $\Omega \subset \mathbb{R}^2$ is such that $H^2$-regularity is granted for the Laplacian with homogeneous Dirichlet data.
This is, for instance, the case when $\Omega$ is bounded and convex or if it is bounded with a sufficiently smooth boundary.
Note, that the time discrete solutions, $\theta^{n+1}$ and $\hu^{n+1}$ are then automatically in $H^2(\Omega)$. 
\end{assumption}
To achieve stability (Lemma~\ref{lem:weak-stability}), the constants $\bar{\alpha}$ and $\bar{C}$ have to be chosen appropriately:
\begin{assumption}[Constants]\label{ass:constants}
    For the previously defined energy, we assume
    \begin{align*}
        \bar\alpha := 4 \sqrt{2} \alpha
        \quad\textmd{ and } \quad
        \bar{C} := \max(
        16 C_{\hf_1}^2,
        8 \bar\alpha^2 C_g^2,
        1
        ),
    \end{align*}
    which depend on the Lipschitz constant $\alpha$ for $\bff_2$ and the bounds $C_{\bff_1}$, $C_{g}$ for $\bff_1$ and $g$.
\end{assumption}
The following constraints on the regularity of the continuous solution are needed to attain the second-order convergence in time in our proofs. 
\begin{assumption}[Regularity solution]\label{ass:solution}
    Assume that the Boussinesq equation has a solution $(\hu, p, \theta)$ satisfying
    \begin{align}
         & \partial^2_t{\hu} \in L^2\big((0,T); (H^2_0(\Omega))^\dim\big),
        \quad
        \partial^3_t\hu \in L^2\big((0,T); (L^2(\Omega))^\dim\big),
        \label{eq:error-assumption-u}
        &\textmd{for the velocity,}\\
                 & \partial^2_t{p} \in L^2\big((0,T); (H^1(\Omega) \cap L^2_0(\Omega))^\dim\big),
        \label{eq:error-assumption-p}
        &\textmd{for the pressure,} \\
                 & \partial^2_t{\theta} \in L^2\big((0,T); H^1_0(\Omega)\big),
        \quad
        \partial^3_t \theta \in L^2\big((0,T); L^2(\Omega)\big),
        \label{eq:error-assumption-T}
        &\textmd{for the temperature, and} \\
         & 
         \partial^2_t{\hf_1},
         \partial^2_t{(\hf_2 \circ \theta)} 
         \in L^2\big((0,T); L^2(\Omega)\big),
        \label{eq:error-assumption-rhs}
        &\textmd{for the right-hand side.}
            \end{align}
\end{assumption}
Note that for a right-hand side given in Eq.~\eqref{eq:boussinesq:momentum} the conditions in Eq.~\eqref{eq:error-assumption-rhs} directly follow from Eq.~\eqref{eq:error-assumption-T}.
For the error stability and error estimate, we will always assume a stable, second-order scheme for bootstrapping the calculation:
\begin{assumption}[Stable bootstrapping algorithm]\label{ass:bootstrapping}
    We suppose that we can find a constant $C_{b} > 0 $ independent of $\tau$, such that at $t = \tau$ we have
    \begin{align*}
        \norm{\nabla \theta^1}^2, \norm{\nabla \bu^1}^2, 
        \norm{\nabla p^1}^2,
        \norm{\Delta \theta^1}^2, \norm{\Delta \bu^1}^2 & \leq C_{b},
        \\
        \norm{\nabla (\theta^1 - \theta^1_\star)}^2, \norm{\nabla (\bu^1 - \hu^1_\star)}^2, \norm{\nabla (p^1 -p^1_\star)}^2                        & \leq C_b \tau^4
        ,
        \textmd{ and }
        \\
        \norm{\Delta (\theta^1 - \theta^1_\star)}^2, \norm{\Delta (\bu^1 - \hu^1_\star)}^2                        & \leq C_{b} \tau^3
        .
    \end{align*}
\end{assumption}
In practice, this can be achieved by calculating $\theta^1$ with a single-step method, which is locally second-order accurate.

In the following theory section, we will prove our main result:
\begin{theorem}\label{thm:main-error-result}
    Let $l\geq1$, {and} $k\geq 3$. 
    {Under Assumptions \ref{ass:domain}, \ref{ass:constants}, \ref{ass:solution}, and \ref{ass:bootstrapping},}
    we find for Scheme~\ref{scheme:fem} a constant $C>0$ independent of $\tau$, such that for all $m\in\Ii$
    \[
        \norm{\hu^{m+1}}, r^{m+1}, \eta^{m+1}, \xi^{m+1} \leq C
    \]
    holds.
    Further,
    for sufficiently small $\tau>0$ and for all $m\in\Ii$, we have
    \[
        \norm{\nabla (\bu^m - \hu^m_\star)}^2
        + \norm{\nabla (\theta^m - \theta^m_\star)}^2
        + \tau \sum_{n=1}^{m-1} \norm{\Delta (\bu^n - \hu^n_\star)}^2
        + \tau \sum_{n=1}^{m-1} \norm{\Delta (\theta^n - \theta^n_\star)}^2
        \leq C \tau^4
        .
    \]
\end{theorem}
\section{A priori analysis of the time integration scheme}\label{sec:theory}
Note that we restrict the analysis to the special case of $d=2$.
We define the errors for the scaled and unscaled velocities $\be_\hu^{n+1} = \bu^{n+1} - \hu^{n+1}_\star$, $\he_\hu^{n+1} = \hu^{n+1} - \hu^{n+1}_\star$, the pressure $e_p^{n+1} = p^{n+1} - p^{n+1}_\star$, the temperature $e_\theta^{n+1} = \theta^{n+1} - \theta^{n+1}_\star$, and the auxiliary variable $e^{n+1}_r = r^{n+1} - r^{n+1}_\star$, and give a short overview of the proof, which follows the strategy in \cite{huang2023stability}, but requires additional techniques due to the coupling and the reformulation suitable for finite element in space approaches:
We will start by stating the main tools of every error analysis, i.e., Gr\"onwall's inequality in Lemma~\ref{lem:Gronwall-reformulated}, and an estimate for the pressure in Lemma~\ref{lem:stokes-technical}. After that we will summarize the approximation quality of the operators $\delta^k$ and $\tD^k$ for sufficiently regular operators in Lemma~\ref{lem:time-derivative-and-extrapolation}, followed by simple summation formulas in Lemma~\ref{lem:abb-equation-2}, when testing our equation with the corresponding test function.
Starting from this,  we will obtain a very general weak stability result in Lemma~\ref{lem:weak-stability}.
We will then proceed with a stronger stability result for the temperature in Lemma~\ref{lem:stability-heat} and the velocity in Lemma~\ref{lem:stability-stokes}.
In the error estimates for $\be_\hu^n$, $e_\theta^n$, and $e^{n+1}_r$ in Lemma~\ref{lem:error-boussinesq} and \ref{lem:error-aux}, we will have to assume in addition that $\eta^n$ converges quadratically with the time step to $1$, i.e., $\abs{1-\eta^n} \leq C_0 \tau^2$, and that we have lower bounds for $\xi^n > c_\xi$, and $\eta^n > c_\eta$, which cannot be found in general.
Finally, in Lemma~\ref{lem:boussinesq_etastability}, we show that the previously introduced constraints on $\eta^n$ and $\xi^n$ are granted for small enough $\tau$. Hence, Theorem~\ref{thm:main-error-result} will follow as a simple corollary of all the previous results.

\begin{remark}
In the following estimates, we use a generic constant $C$ and highlight dependencies on other constants with subscripts. If the generic constant in $a\leq Cb$ is not of interest, we abbreviate via $a \lesssim b$.
\end{remark}

\subsection{Preliminaries}
For convenience, we state the following variant of Gr\"onwall's lemma:
\begin{lemma}\label{lem:Gronwall-reformulated}
    Given two Banach spaces $U, V$ and two sequences $(u^n)_{n\in\Nn} \in l^{\infty}(U)$, $(v^n)_{n\in\Nn} \in l^\infty(V)$. Assume that there are $a_n,\,b_n,\,c_n,\,d_n,\,A,\,C,\,\tau,\,l,\,k \geq 0$ for $n=1,\ldots,m-1$, such that
    \begin{align*}
        A+ \norm{ u^m }^2_U + \norm{ v^m }^2_V \leq C  + \tau  \sum_{n=1}^{m-1} \big(a_n \norm{ u^n }^2_U+b_n \norm{ v^n }^2_V+c_n \norm{ \delta^{l+1} u^n }^2_U+d_n \norm{ \delta^{k+1} u^n }^2_V\big) ,
    \end{align*}
    and that there are $M_a,\,M_b,\,M_c,\,M_d \geq 0$ such that
    \[
        \tau \sum_{n=1}^{m-1} a_n \leq M_a,\quad \tau \sum_{n=1}^{m-1} b_n \leq M_b,\quad \tau \sum_{n=1}^{m-1} c_n \leq M_c,\quad \textmd{ and }\quad \tau \sum_{n=1}^{m-1} d_n \leq M_d
        .
    \]
    Then
    \[
        \norm{ u^m }^2_U + \norm{ v^m }^2_V + A \leq C \exp(M_a + M_b  + (2l+1)^2 M_c  + (2k+1)^2 M_d).
    \]
\end{lemma}
\begin{proof}
    This is a direct consequence of~\cite[Lemma~1]{huang2023stability}.
\end{proof}

For estimating the pressure equation, we cite the following technical results:
\begin{lemma}\label{lem:stokes-technical}
    Given $\hu \in (H^2(\Omega))^\dim \cap (H^1_0(\Omega))^\dim$    . Then the Stokes pressure $p_s(\hu)$ is given by
    \[
        \inner{\nabla p_s(\hu), \nabla q} = - \inner{ \nabla \times \nabla \times \hu , \nabla q }
    \]
    and for every $\varepsilon_s>0$ there is a $C>0$, such that
    \[
        \norm{ \nabla p_s(\hu) }^2 \leq \left( \frac{1}{2} +\epsilon_s \right) \norm{ \Delta \hu }^2 + C \norm{ \nabla \hu }^2 .
    \]
\end{lemma}
\begin{proof}
    See~\cite[Sec.~2.1]{liu2007stability} for the definition, and \cite[Thm.~1.2]{liu2007stability} for the estimate.
\end{proof}
\begin{remark}
The Stokes pressure $p_s$ is the pressure component due to the viscosity of the operator omitting advective effects and external forces.
\end{remark}

\begin{lemma}[Ladyzhenskaya's inequality]\label{lem:vectorial-ladyzhenskaya-inequality}
    Let $\Omega$ be as above and $\hu \in (H^1(\Omega))^{\dim}$, then there is a $C>0$ such that
    $ \norm{\hu}_{L^4(\Omega)} \leq C \norm{\hu}^{1/2} \norm{\hu}^{1/2}_{H^1(\Omega)}.$
\end{lemma}
\begin{remark}
    Note that $\norm{\hu}_{H^1(\Omega)} \lesssim \norm{\nabla \hu}$ for $\hu \in (H^1_0(\Omega))^\dim$ by Poincar\'e's inequality, and that  $\norm{\hu}_{H^2(\Omega)} \lesssim \norm{\Delta \hu}$ for $\hu \in (H^2(\Omega) \cap H^1_0(\Omega))^\dim$ follows from the assumed elliptic regularity.
\end{remark}

We now state the following consistency inequalities:
\begin{lemma}\label{lem:time-derivative-and-extrapolation}
    Given $v \in C^1([0,T^*], V)$ with a Hilbert space $V$ with $T^* \geq T + k \tau$ sufficiently large. Then 
    $\delta^{k} v^{n+1}_\star = v^{n+k}_\star + \zeta_{n+1,k}(v)$, 
    $\tD^k v^{n+1}_\star           = 2 \tau \partial_t{v}^{n+k}_\star + \Xi_{n+1}(v), $
    where 
    $ \sum_{n=0}^{m-1} \norm{ \zeta_{n+1,k}(v) }^2_V \leq C\, \tau^3$ if  $\partial^2_t{v} \in L^2\big((0,T); V\big) $, {and}
    $\sum_{n=1}^{m-1} \norm{ \Xi_{n+1}(v) }^2_V \leq C\, \tau^5$ if $\partial^3_t v \in L^2\big((0,T); V\big)$ with constants $C$ not depending on $\tau$.
\end{lemma}\begin{proof}See the \proofref{lem:time-derivative-and-extrapolation} in Appendix~\ref{sec:app:proofs}, using the remainder term in Taylor's theorem.\end{proof}
The following two inequalities will be central to our estimates:
The first is an extension of Eq.~(3.16) from \cite{huang2023stability} for general $k$.
The second one is a variant of Eq.~(3.8) from \cite{huang2023stability} but with more favorable constants with respect to $k$. Both estimates will allow us later to show the validity of the approach for all $k\geq3$.
\begin{lemma}
    \label{lem:abb-equation-2}
    Given a sequence $(v_n)_{n \in \Zz} \in l^\infty(V)$, in some Hilbert space $V$ with inner product $\inner{\cdot, \cdot}_V$ and induced norm $\norm{\cdot}_V$, the following inequalities hold for $k\geq\nicefrac{1}{2}$:
    \begin{align}
        \inner{\tD^k v^{n+1}, \delta^{k+1} v^{n+1}}_V
        &
        \geq 
        a_k(\norm{v^{n+1}}^2_V - \norm{v^n}^2_V)
        +
        \norm{b_k v^{n+1} - c_k v^n}^2_V
        -\norm{b_k v^n - c_k v^{n-1}}^2_V
        \label{eq:factorizationD}
    \end{align}
    with the constants 
        $a_k = \frac{3}{2 (k+1)}$, 
        $ c_k = \sqrt{k^{2} + \frac{k}{2} - \frac{1}{2}}$,
        and 
        $b_k = c^{-1}_k\left(k^{2} + \frac{3 k}{2} - 1\right)$.
                                Further,
    \begin{align}
        \inner{ \delta^k v^{n+1}, \delta^{k+1} v^{n+1} }_V
        \geq 
        \hat{a}_k
        \norm{\delta^{k+1}v^{n+1}}^2_V
        +  \left(\hat{d}_k + \hat{f}_k\right) \norm{v^{n+1}}^{2}_V
        - \hat{d}_k \norm{v^{n}}^{2}_V
        ,
        \label{eq:factorizationextr2}
    \end{align}
   with constants $\hat{a}_k= \frac{4 k^{2} - 1 - \epsilon}{4 k^{2} + 4 k + 1}$, $\hat{d}_k = \frac{k + \frac{1}{2} - \epsilon k}{\left(2 k + 1\right)}$, and $\hat{f}_k = \epsilon$, where $\epsilon > 0$ is sufficiently small.
                            \end{lemma}
\begin{proof}See the \proofref{lem:abb-equation-2} in Appendix~\ref{sec:app:proofs}.\end{proof}
\begin{remark}
    The main advantage of inequality~\eqref{eq:factorizationextr2} to (3.18) in \cite{huang2023stability} is the slightly larger coefficient before the term $\norm{\delta^{k+1} v^{n+1}}^2$, which can be observed if we compare the numbers given in the following tables with the results reported in \cite{huang2023stability}:
    \begin{center}
        \begin{tabular}{l|cccccc}
            $k$       & 1     & 2     & 3     & 4     & 5     & $\infty$ \\
            \hline
            $\hat{a}_k|_{\epsilon=0}$ & 1/3 & 3/5 & 5/7 & 7/9 & 9/11 & 1
            \\
            $\hat{a}_k|_{\epsilon=1/k^2}$ & 0.222 & 0.590 & 0.712 & 0.777 & 0.818 & 1
                                            \end{tabular}
    \end{center}
    This will allow us to prove error estimates for smaller $k$.
    For simplicity, we will always use $\epsilon=1/k^2$ in our proofs, which will turn out to be sufficiently small. The limit $\epsilon=0$ does not yield a smaller $k$ compared to $\epsilon=\nicefrac{1}{k^2}$.
\end{remark}

\subsection{Stability estimates}

As a first step in our estimate, we state the following weak stability result for the velocity and auxiliary variable:
\begin{lemma}[Weak stability]\label{lem:weak-stability}
    Assume that the constants of the energy functional $\mcE$ are set {according to Assumption~\ref{ass:constants}} to
                                        $\bar\alpha = 4 \sqrt{2} \alpha$ and
    $\bar{C} = \max( 16 C_{\hf_1}^2, 8 \bar\alpha^2 C_g^2, 1 )$,
    depending on the Lipschitz constant $\alpha$ for $\bff_2$ and the bounds $C_{\bff_1}$, $C_{g}$ for $\bff_1$ and $g$.
    Further, let $\bu^{n+1}$ and $\hu^{n+1}$ be the solutions of Scheme~\ref{scheme:fem}.
    Then, there is an $M$ that may depend on $T$ and $r_0 = \nicefrac{1}{2}\norm{\hu(0)}^2 + \nicefrac{\bar{\alpha}^2}{2}\norm{\theta(0)}^2 + \bar{C}$, such that for all $m+1\leq N$ we have
    \begin{align*}
        0 \leq  r^{m+1} \leq M,
        \;
        0 \leq \xi^{m+1} \leq M,
        \;
        \abs{\eta^{m+1}} \leq M,
        \textmd{ and }
        \\
        \nu \tau \sum_{n=0}^{m} \xi^{n+1} \norm{\nabla \bu^{n+1}}^2 \leq M,\,
        \nu \tau \sum_{n=0}^{m} \norm{\nabla \hu^{n+1}}^2 \leq M,
        \norm{\hu^{m+1}}^2 \leq M.
    \end{align*}
\end{lemma}
\begin{proof}
    The proof is a straightforward adaption of \cite[Thm.~6]{huang2023stability} to our time integrator.
    Recursively applying our time integrator gives us
    \begin{align}
    r^{m+1}
    =
    \exp\left(
    \tau
    \sum_{j=0}^m
        \frac{\nicefrac{d\mcE}{dt}(\theta^{j+1}, \bu^{j+1})}{\mcE(\theta^{j+1}, \bu^{j+1}) + \bar{C}}
    \right)
    r^0,
    \label{eq:weak-stability-explicit-r}
    \end{align}
    which guarantees $r^{m+1} \geq 0$ when $r^0 \geq 0$ and also $\xi^{m+1} \geq 0$ by \eqref{eq:scheme-eta-xi}.
    Next, we want to bound $r^{m+1}$ from above.
    By omitting the negative terms of the exponential, we obtain
    \begin{align}
    r^{m+1}
    \leq
    \exp\left(
    \tau
    \sum_{n=0}^m
    \left(
        \frac{\inner{\hf(t^{n+1}, \theta^{n+1}), \bu^{n+1}}}{\mathcal{E}(\theta^{n+1}, \bu^{n+1}) + \bar{C}}
        +
        \frac{\bar\alpha^2 \inner{g(t^{n+1}), \theta^{n+1}}}{\mathcal{E}(\theta^{n+1}, \bu^{n+1}) + \bar{C}}
    \right)
    \right)
    r^0.
    \label{eq:estimate-above-r-1}
    \end{align}
    To further estimate the exponential function, 
    we use Young's inequality, and by our choice of $\bar{\alpha}$ and $\bar{C}$, we obtain
    \begin{align}
        \abs{
            \frac{\inner{\hf(t^{n+1}, \theta^{n+1}), \bu^{n+1}}}{\mathcal{E}(\theta^{n+1}, \bu^{n+1}) + \bar{C}}
        }
        &
        \leq 
        \frac{
            (C_{\hf_1} + \alpha\norm{\theta^{n+1}})
            \norm{\bu^{n+1}}
        }{\mathcal{E}(\theta^{n+1}, \bu^{n+1}) + \bar{C}}
        \leq
        \frac{
            \frac{1}{8}
            \norm{\bu^{n+1}}^2
            + 4 \alpha^2 \norm{\theta^{n+1}}^2
            + 4 C_{\hf_1}^2
        }{\mathcal{E}(\theta^{n+1}, \bu^{n+1}) + \bar{C}}
        \nonumber
        \\
        &
        \leq 
        \frac{
            \frac{1}{8}
            \norm{\bu^{n+1}}^2
            +
            \frac{1}{4}
            \frac{\bar{\alpha}^2}{2} \norm{\theta^{n+1}}^2
            +
            \frac{1}{4}\bar{C}
        }{
            \frac{1}{2}
            \norm{\bu^{n+1}}^2
            +
            \frac{\bar{\alpha}^2}{2} \norm{\theta^{n+1}}^2
            +
            \bar{C}
        }
        \leq \frac{1}{4},
        \label{eq:weak-stability-estimate-f}
    \end{align}
    and
    \begin{align}
        \abs{
            \frac{\bar\alpha^2 \inner{g(t^{n+1}), \theta^{n+1}}}{\mathcal{E}(\theta^{n+1}, \bu^{n+1}) + \bar{C}}
            }
        &
        \leq 
        \frac{ \frac{1}{8} \bar\alpha^2  \norm{\theta^{n+1}}^2 + 2 \bar\alpha^2  C_g^2 }{\mathcal{E}(\theta^{n+1}, \bu^{n+1}) + \bar{C}}
        \leq
        \frac{ \frac{1}{8} \bar\alpha^2  \norm{\theta^{n+1}}^2 + \frac{1}{4}\bar{C} }{
            \frac{1}{2}
            \norm{\bu^{n+1}}^2
            +
            \frac{\bar{\alpha}^2}{2} \norm{\theta^{n+1}}^2
            +
            \bar{C}
        }
        \leq
        \frac{1}{4}
        .
        \label{eq:weak-stability-estimate-g}
    \end{align}
    Plugging the last two estimates into \eqref{eq:estimate-above-r-1}, we obtain the $\tau$-independent bounds $r^{n+1}\!\leq\!e^{T/2} r^0\!\!=:\!\!M_r$ and $\xi^{n+1} \leq \exp(T/2) r^0 / \bar{C} =: M_\xi$.

    To estimate the gradients, we 
    multiply Eq.~\eqref{eq:scheme-r} by 
    $ \exp\left(- \tau \tfrac{\nicefrac{d\mcE}{dt}(\theta^{n+1}, \bu^{n+1})}{\mcE(\theta^{n+1}, \bu^{n+1}) + \bar{C}}\right)$ and exploit $\exp(x) \geq 1 + x$, yielding 
    \begin{align*}
    r^{n+1}
    \left(
    1
    -
        \tau
        \frac{\nicefrac{d\mcE}{dt}(\theta^{n+1}, \bu^{n+1})}{\mcE(\theta^{n+1}, \bu^{n+1}) + \bar{C}}
    \right)
    \leq
    r^{n}.
    \end{align*}
    This can be written, by omitting the positive term
    $\kappa \bar{\alpha}^2 \norm{\nabla \theta^{n+1}}^2$ on the left, as 
    \begin{align*}
    r^{n+1}
    -r^n
    + \tau \xi^{n+1} 
    \nu \norm{\nabla \bu^{n+1}}^2
        \leq
    \tau\cdot
    r^{n+1}
    \frac{
    \inner{\hf(t^{n+1}, \theta^{n+1}), \bu^{n+1}}
    +
    \bar\alpha^2 \inner{g(t^{n+1}), \theta^{n+1}}
    }{ \mcE(\theta^{n+1}, \bu^{n+1}) + \bar{C} }
    .
    \end{align*}
    Summing up from $n=0$ to $m$ yields, using the previous estimates for the right-hand-side terms, and omitting the positive $r^{m+1}$ term, the bound
    \begin{align*}
    \tau \nu \sum_{n=0}^{m}\xi^{n+1} 
    \norm{\nabla \bu^{n+1}}^2
        \leq
    r^0 + T M_r/2
    =: M_{\nabla \hu}
    ,
    \end{align*}
    which is again independent of $\tau$.
    Since the inequality $1-|1-x|^\gamma \leq \max(\gamma,1) \cdot x$ holds for all $x,\gamma \geq 0$, we can use $\xi^{n+1} \geq 0$, to obtain 
    \begin{align*}
        \abs{\eta^{n+1}}
        \leq &
        2 \, \xi^{n+1}
        \leq
        2
        \frac{r^{n+1}}{\mcE(\theta^{n+1}, \bu^{n+1}) + \bar{C}}
        \leq
        \frac{2 M_r}{\nicefrac{1}{2} \norm{\bu^{n+1}}^2 + \bar{C}}
        .
    \end{align*}
    This upper bound for $\eta^{n+1}$ allows us to bound $\norm{\hu^{n+1}}^2$ by
        \begin{align*}
        \norm{
            \hu^{n+1}
        }^2
        =
        (\eta^{n+1})^2
        \norm{
            \bu^{n+1}
        }^2
        \leq
        4 M_r M_\xi
        \frac{\nicefrac{1}{2}\norm{\bu^{n+1}}^2}{\nicefrac{1}{2}\norm{\bu^{n+1}}^2 + \bar{C}}
        \leq M_\hu
        ,
    \end{align*}
    where $M_\hu := 4 M_r M_\xi $.
    Finally, since
    \begin{align*}
        \norm{\nabla \hu^{j+1}}^2
        =
        (\eta^{j+1})^2
        \norm{\nabla \bu^{j+1}}^2
        =
        4 M_\xi
        \xi^{j+1}
        \norm{\nabla \bu^{j+1}}^2
        ,
    \end{align*}
    we have $ \nu \tau \sum_{j=0}^m \norm{\nabla \hu^{j+1}}^2 \leq 4 M_\xi M_{\nabla\hu}$, which shows the last inequality. Setting \[M = \max \{M_r, M_\xi, M_\eta, M_\hu, M_{\nabla\hu}, 4 M_\xi M_{\nabla\hu} \},\] 
    the claim follows. 
\end{proof}
\begin{remark}
The constant $M$ strongly depends on the initial energy of the system and, hence, on the initial temperature distribution.
\end{remark}

In the following, we will always assume that $\bar{\alpha}$ and $\bar{C}$ are chosen, such that the conditions of Lemma~\ref{lem:weak-stability} are satisfied, and hence the conditions are not explicitly stated.
A direct consequence of the previous lemma is that if we have a strictly positive lower bound $c_\xi > 0$ for $\xi^{n}$, such that $c_\xi \leq \xi^n$, for all $n \leq m+1$, then $\nu \tau \sum_{n=0}^m \norm{\nabla \bu^{n+1}}^2 \leq \frac{M}{c_\xi} =: C_1$.
Similarly, if we have a strictly lower bound $c_\eta > 0$ for $\eta^{n}$, s.t. $c_\eta \leq \eta^n$, for all $n \leq m+1$, then
$ \norm{\bu^{n}}^2 \leq \frac{M}{c_\eta} =: M_{\bu}$.
The lower bounds $c_\xi$ and $c_\eta$ cannot be satisfied for an arbitrary time step size, but we will show that they are satisfied for small enough time steps.

To derive stability estimates for $\theta$ and $\hu$, we need the following result to estimate the advective terms:
\begin{lemma}\label{lem:estimate-advection-term}
    For $\hu \in( H^1_0(\Omega))^\dim$ and $\theta, \tilde{\theta} \in H^2(\Omega) \cap H^1_0(\Omega)$, let $\norm{\hu} \leq \widetilde{M}$, then for every $\epsilon_1, \epsilon_2 > 0$ there is a $C=C(\epsilon_1,\epsilon_2)>0$ such that
    \begin{align}
        \abs{
            \inner{(\hu \cdot \nabla) \theta, \Delta \tilde \theta}
        }
         & \leq
        \epsilon_1 \norm{\Delta \tilde \theta}^2
        +\epsilon_2 \norm{\Delta \theta}^2
        + C \norm{\nabla \hu}^2 \norm{\nabla \theta}^2
        .
        \label{eq:estimate-advection-term-ut}
    \end{align}
    For
    $\hu, \tilde{\hu}, \hat{\hu} \in( H^1_0(\Omega))^\dim$ with $\tilde{\hu},\hat{\hu} \in (H^2(\Omega))^\dim$, and $\epsilon_1, \epsilon_2 > 0$ there is a $C=C(\epsilon_1,\epsilon_2)>0$ with
    \begin{align}
        \abs{ \inner{(\hu \cdot \nabla \hat{\hu}), \Delta \tilde{\hu}} }
         & \leq
        \epsilon_1 \norm{\Delta \tilde{\hu}}^2
        +\epsilon_2 \norm{\Delta \hat{\hu}}^2
        + C \norm{\nabla \hu}^2 \norm{\nabla \hat{\hu}}^2
        \label{eq:estimate-advection-term-u}
        .
    \end{align}

\end{lemma}
\begin{proof}
    From Cauchy-Schwarz, Lemma~\ref{lem:vectorial-ladyzhenskaya-inequality}, and the generalized Young inequality, we obtain
    \begin{align*}
        \abs{
            \inner{(\hu \cdot \nabla) \theta, \Delta \tilde \theta}
        }
         & \leq
        \norm{\hu \cdot \nabla \theta}\norm{\Delta \tilde \theta}
        \leq
        \norm{\hu}_{L^4(\Omega)} \norm{\nabla \theta}_{L^4(\Omega)} \norm{\Delta \tilde \theta}
        \\
         & \leq
        \norm{\hu}^{1/2}
        \norm{\nabla \hu}^{1/2}
        (
        \norm{\nabla \theta}
        +
        \norm{\nabla \theta}^{1/2}
        \norm{\Delta \theta}^{1/2}
        )
        \norm{\Delta \tilde \theta}
        \\
         & \leq
        \epsilon_1 \norm{\Delta \tilde \theta}^2
        +
        C
        (
        \norm{\hu}
        \norm{\nabla \hu}
        \norm{\nabla \theta}^2
        +
        \norm{\hu}
        \norm{\nabla \hu}
        \norm{\nabla \theta}
        \norm{\Delta \theta}
        )
        .
    \end{align*}
    Using Poincar\'e on the second term $\norm{\hu} \lesssim \norm{\nabla\hu}$, $\norm{\hu}\leq \widetilde{M}$ on the third term, and separating its components with Young yield Eq.~\eqref{eq:estimate-advection-term-ut}. The vectorial case in Eq.~\eqref{eq:estimate-advection-term-u} follows analogously.
\end{proof}

{Assumptions \ref{ass:domain}, \ref{ass:constants}, \ref{ass:solution}, and \ref{ass:bootstrapping} will be vital for nearly all of our coming proofs. 
Hence, we will assume they hold from now on and refrain from mentioning them explicitly in the coming theorems.}

The next Lemma shows that we do not need to scale $\theta^m$ with an auxiliary variable to achieve stability, as long as $\hu^n$ satisfies the weak stability result of Lemma~\ref{lem:weak-stability}.
\begin{lemma}[Stability estimate for the heat equation]
    \label{lem:stability-heat}
    Given Assumptions~\ref{ass:domain}, \ref{ass:constants}, \ref{ass:solution}, and \ref{ass:bootstrapping}, we have
    \begin{align*}
        \norm{\nabla \theta^m}^2
        + \sum_{n=0}^{m} \tau \norm{\Delta \theta^n}^2
        \leq C_\theta^2,
    \end{align*}
    where $C_\theta$ depends on $M$, $T$ but not on $\tau$.
\end{lemma}
\begin{proof}
    Testing Eq.~\eqref{eq:scheme-T} with $(-\Delta)\delta^{l+1}\theta^{n+1}$ and integrating by parts yield
    \begin{align*}
        \inner{\nabla\tD^l \theta^{n+1}, \nabla \delta^{l+1} \theta^{n+1}} &
        + 2 \tau \kappa \inner{\Delta \delta^{l} \theta^{n+1}, \Delta \delta^{l+1} \theta^{n+1}}
        =
        2 \tau \inner{g(t^{n+l}), \Delta \delta^{l+1} \theta^{n+1}}
        \\ &
        - 2 \tau \inner{(\delta^{l+1} \hu^{n} \cdot \nabla)\delta^{l+1} \theta^{n}, (-\Delta) \delta^{l+1} \theta^{n+1}}
        =: (I) + (II)
        ,
    \end{align*}
    where we abbreviate the right-hand side and advection terms by $(I)$ and $(II)$.
    We estimate the right-hand side term by
    \begin{align}
                \abs{(I)}
        & 
        \leq 2 \tau C\norm{g(t^{n+l})}^2
        +
        2 \tau \kappa \epsilon
        \norm{\Delta \delta^{l+1} \theta^{n+1}}^2
                \leq
        2 \tau C + 2 \tau \kappa \epsilon  \norm{\Delta \delta^{l+1} \theta^{n+1}}^2,
        \label{eq:stabheat-i}
    \end{align}
    and the advection term with Eq.~\eqref{eq:estimate-advection-term-ut} in Lemma~\ref{lem:estimate-advection-term} using $\norm{\hu^n}^2 \leq M$ by Lemma~\ref{lem:weak-stability} to
    \begin{align}
        \abs{(II)}
        &
        \leq 
        2 \tau\kappa
        \epsilon
        \norm{
            \Delta \delta^{l+1} \theta^{n+1}
        }^2
        +
        2 \tau\kappa
        \epsilon
        \norm{
            \Delta \delta^{l+1} \theta^{n}
        }^2
        +
        2 \tau
        C
        \norm{\nabla \delta^{l+1}\hu^{n}}^2
        \norm{\nabla \delta^{l+1} \theta^n}^2
        .
        \label{eq:stabheat-ii}
    \end{align}
    Collecting the previous estimates in Eq.~\eqref{eq:stabheat-i} and \eqref{eq:stabheat-ii}, using Lemma~\ref{lem:abb-equation-2}, 
    summing up from $n=1$ to $m-1$, and omitting some positive terms yield
    \begin{align*}
         &
        a_l \norm{\nabla \theta^{m}}^2
        + 2 \tau \kappa (\hat{a}_l - 3\epsilon)
        \sum_{n=1}^{m-1}
        \norm{\delta^{l+1} \Delta \theta^{n+1}}^2
        + 2 \tau \kappa \hat{f}_l
        \sum_{n=1}^{m-1}
        \norm{\Delta \theta^{n+1}}^2
        \\ &
        \leq
        2 \tau
        C
        \sum_{n=1}^{m-1}
        \norm{\nabla \delta^{l+1}\hu^{n}}^2
        \norm{\nabla \delta^{l+1} \theta^n}^2
        + 2 T C
        +\norm{b_l \nabla \theta^{1} - c_l \nabla \theta^{0}}^2
        + 2 \tau \kappa \hat{d}_l\norm{\Delta \theta^{1}}^2
        + a_l \norm{\nabla \theta^{1}}^2
    \end{align*}
    Setting $\epsilon = \hat{a}_l/4$, using the Gr\"onwall Lemma~\ref{lem:Gronwall-reformulated} with $\theta^n$ as $u^n$ and $0$ as $v^n$, we
    trivially obtain $ \tau \sum a_n =  \tau \sum b_n =  \tau \sum d_n =  0 = M_a = M_b = M_d $ and
    $ \tau \sum c_n = 2 \tau \frac{C}{a_l} \sum_{n=1}^{m-1} \norm{\nabla \delta^{l+1} \hu^n}^2 \leq M_c$,
    where we used Lemma~\ref{lem:weak-stability}.
    With our assumption on a stable bootstrapping algorithm, i.e., $\norm{b_l\nabla \theta^1}^2,\norm{\Delta \theta^1}^2 \leq C$ with $C$ independent of $\tau$, we get the claim.
\end{proof}
Next, we aim for a similar estimate for the Stokes equation with the same techniques displayed in the previous proof. The main challenge will be to estimate the pressure gradient of Eq.~\eqref{eq:scheme-u} with sufficiently small estimated constants in front of $\norm{\Delta\delta^{k+1}\bu^{n+1}}^2$ to match the size of the constant $\hat{a}_k$ from Lemma~\ref{lem:abb-equation-2}. This estimate will be provided in the following Lemma:
\begin{lemma}[Pressure estimate]\label{lem:pressure-stability-estimate}
Given a solution of {Scheme~\ref{scheme:fem}}, satisfying Assumptions~\ref{ass:domain}, \ref{ass:constants}, \ref{ass:solution}, and \ref{ass:bootstrapping}.
Then, there is a constant $C>0$ independent of $m$ and the time step width $\tau$ such that
\begin{align*}
\sum_{n=1}^{m-1}
&
\inner{ \nabla \delta^{k+1} p^n , \Delta \delta^{k+1} \bu^{n+1} }
\leq
\nu
\sqrt{\frac{\nicefrac{1}{2}+\epsilon}{1-\epsilon}}
\sum_{n=1}^{m-1}
\norm{
\Delta \delta^{k+1} \bu^{n+1}
}^2
+
\nu \epsilon \sum_{n=1}^{m-1} \norm{ \Delta \bu^{n} }^2
+ C  \sum_{n=1}^{m-2} \norm{\delta^{k+1}\theta^{n}}^2
\nonumber
\\&
+ C 
\sum_{n=2}^{m-1} 
\norm{ \nabla \delta^{k+1} \bu^{n} }^2
+ C 
\sum_{n=1}^{m-1} 
\norm{ \nabla \hu^{n} }^2
\norm{ \nabla \bu^{n} }^2
+ C (1+ N )
.
\end{align*}
\end{lemma}
\begin{proof}
For all $n\geq1$, we can apply $\nabla$ to Eq.~\eqref{eq:p-rec-definition}, test with an arbitrary $\nabla q$, $q \in H^1(\Omega)$, and eliminate $\nabla \psi^{n+1}$ by Eq.~\eqref{eq:psi-definition} to obtain:
\begin{align*}
\inner{ \nabla \delta^{k} p^{n+1}, \nabla q }
&=
\inner{ - \nu \nabla \nabla \cdot \delta^{k} \bu^{n+1} + \nabla \delta^{k+1} p^n + \frac{1}{2\tau} \tD^k \bu^{n+1} , \nabla q } 
.
\intertext{Further, using Eq.~\eqref{eq:scheme-u}, and introducing the variable $\hh^n = \hf(t^{n+k}, \delta^{k+1} \theta^n) - (\delta^{k+1} \hu^n \cdot \nabla) \delta^{k+1} \hu^n$ to collect terms yield}
\inner{ \nabla \delta^{k} p^{n+1}, \nabla q }&=
\inner{ \nu(\Delta - \nabla \nabla \cdot) \delta^{k} \bu^{n+1} 
- (\delta^{k+1} \hu^n \cdot \nabla) \delta^{k+1} \hu^n
+ \hf(t^{n+k}, \delta^{k+1} \theta^n)
, \nabla q}
\\
&=
\inner{ \nu(\Delta - \nabla \nabla \cdot) \delta^{k} \bu^{n+1} 
+ \hh^n
, \nabla q}
.
\end{align*}
Separating terms at time $t^{n+1}$ from $t^n$ yields
\begin{align}
\langle \nabla p^{n+1}
&
- \nu(\Delta - \nabla \nabla \cdot) \bu^{n+1}
, \nabla q \rangle
=
\frac{k-1}{k}
\inner{ 
\nabla p^n
-
\nu(\Delta - \nabla \nabla \cdot) \bu^{n}, \nabla q}
+
\frac{1}{k}
\inner{
\hh^n
, \nabla q}.
\label{eq:pressure-before-recursion}
\end{align}
Hence, by applying the operator $\delta^{k+1}$, we obtain for all $n\geq2$ the relation
\begin{align*}
\langle \nabla& \delta^{k+1} p^{n+1}
- \nabla p_s(\nu\delta^{k+1}\bu^{n+1})
, \nabla q \rangle
=
\frac{k-1}{k}
\inner{ 
\nabla \delta^{k+1} p^n - \nabla p_s(\nu\delta^{k+1}\bu^{n}), \nabla q}
+
\frac{1}{k}
\inner{
\delta^{k+1}\hh^n
, \nabla q},
\intertext{
which is a recursive equation in the term $\beta_n:=\nabla \delta^{k+1} p^n - \nabla p_s(\nu\delta^{k+1}\bu^{n})$.
Expanding the recursion yields
}
&\inner{\beta_{n+1},\,\nabla q}=
\left( \frac{k-1}{k} \right)^{n-1}\inner{\beta_2, \nabla q}
+  
\frac{1}{k}
\sum_{i=2}^{n} \left( \frac{k-1}{k} \right)^{n-i} \inner{ \delta^{k+1} \hh^i , \nabla q}
.
\end{align*}
We can further simplify
\begin{align*}
\inner{\beta_2,\nabla q} =
\langle \nabla \delta^{k+1} p^{2}
- \nabla p_s(\nu\delta^{k+1}\bu^{2})
, \nabla q \rangle
=
\frac{1}{k}
\langle 
\nabla p_s(\nu\bu^{1})
-
\nabla p^{1}
, \nabla q \rangle
+
\frac{k+1}{k}
\inner{
\hh^1
, \nabla q
}
,
\end{align*}
and hence if we set $\hh^0 = 0$, and use $k-1\geq1$, we obtain
\begin{align}
\nonumber
\langle \nabla \delta^{k+1}p^{n+1}
, \nabla q \rangle
&=
\inner{
\nabla  p_s (\nu\delta^{k+1} \bu^{n+1}), \nabla q}
+
\frac{1}{k}
\left(
\frac{k-1}{k}
\right)^{n-1}
\inner{ 
\nabla p_s(\nu \bu^{1})
+
\nabla 
p^1
, \nabla q} \\&\quad
+ 
\frac{1}{k}
\sum_{i=1}^{n}
\left( \frac{k-1}{k} \right)^{n-i}
\inner{
\delta^{k+1}
\hh^i
, \nabla q}
.
\label{eq:pressure-est15}
\end{align}
Note that the geometric series $ \sum_{i=0}^{n-1} \left( \frac{k-1}{k} \right)^{i} \leq k$, $k\geq1$, can be bounded independently of $n$ by $k$. Hence, using $q=\delta^{k+1} p^{n+1}$,
and applying Young's inequality yield 
\begin{align*}
\frac{1}{2}( 1- \epsilon )
&
\norm{\nabla \delta^{k+1}p^{n+1}
}^2
\leq
\frac{1}{2}
\norm{
\nabla  p_s (\nu\delta^{k+1} \bu^{n+1})}^2
+
\left( \frac{k-1}{k} \right)^{2n}
C
\norm{
\nabla 
p^1
- \nabla p_s(\nu  \bu^{1})}^2
\\&
+ 
C
\sum_{i=1}^{n}
\left( \frac{k-1}{k} \right)^{n+1-i}
\norm{
\delta^{k+1}
\hh^i
}^2
.
\end{align*}
Summing from $n=1$ to $m-1$ further yields
\begin{align*}
\sum_{n=1}^{m-1}
\norm{\nabla \delta^{k+1}p^{n+1}
}^2
&
\leq
\nu^2
\frac{\nicefrac{1}{2}+\epsilon}{1-\epsilon}
\sum_{n=1}^{m-1}
\norm{
\Delta \delta^{k+1} \bu^{n+1}
}^2
+
C
\sum_{n=1}^{m-1}
\norm{
\nabla \delta^{k+1} \bu^{n+1}
}^2
+
C
\sum_{i=1}^{m-1}
\norm{
\delta^{k+1}
\hh^i
}^2
,
\end{align*}
where we further used $\norm{ \nabla  p^1 - \nabla p_s(\nu \bu^{1})}^2 \leq C$ due to our stable initialization procedure, the boundedness of the geometric series, and applied Lemma~\ref{lem:stokes-technical} to estimate $p_s(\nu \delta^{k+1}\bu^{n+1})$.
It remains to show the boundedness of $\norm{\delta^{k+1}\hh^i}$.
With Ladyzhenskaya's and Young's inequality, in combination with $\norm{\delta^{k+1} \bu^n} \leq C$, we obtain
\begin{align*}
\norm{
\hh^n
}^2
&
\leq
C
+
\alpha^2 C \norm{\delta^{k+1}\theta^n}^2
+
\norm{
\delta^{k+1} \hu^n
}^2_{L^4(\Omega)}
\norm{
\nabla \delta^{k+1}
\hu^n
}^2_{L^4(\Omega)}
\\
&\leq
C
+
\alpha^2 C \norm{\delta^{k+1}\theta^n}^2
+
C
\norm{
\nabla \delta^{k+1}
\hu^n
}^4
+
\nu^2
\tilde{\epsilon}
\norm{
\Delta \delta^{k+1}
\hu^n
}^2
\\
&\leq
C
+
\alpha^2 C \norm{\delta^{k+1}\theta^n}^2
+
C
\sum_{i=0}^1
\left(
\norm{
\nabla
\hu^{n-i}
}^2
\norm{
\nabla
\bu^{n-i}
}^2
+
\nu^2
\tilde{\epsilon}
\norm{
\Delta 
\bu^{n-i}
}^2
\right)
,
\end{align*}
where we introduced the parameter $\tilde\epsilon > 0$.
Hence we have
\begin{align}
\sum_{n=1}^{m-1}
&
\norm{\nabla \delta^{k+1}p^{n+1}
}^2
\leq
\nu^2
\frac{\nicefrac{1}{2}+\epsilon}{1-\epsilon}
\sum_{n=1}^{m-1}
\norm{
\Delta \delta^{k+1} \bu^{n+1}
}^2
+\nu^2 \tilde{\epsilon}C
\sum_{n=0}^{m-1}
\norm{
\Delta \bu^{n}
}^2
+ C  \sum_{n=1}^{m-1} \norm{\delta^{k+1}\theta^n}^2
\nonumber
\\&
+ C \sum_{n=1}^{m-1} \norm{ \nabla \delta^{k+1} \bu^{n+1} }^2
+ C \sum_{n=0}^{m-1} 
\norm{ \nabla \hu^{n} }^2
\norm{ \nabla \bu^{n} }^2
+ C (1+ N )
.
\label{eq:est14}
\end{align}
The final estimate now follows immediately from Cauchy-Schwarz and Young by
\begin{align*}
\inner{ \nabla \delta^{k+1} p^n , \Delta \delta^{k+1} \bu^{n+1} }
\leq
\frac{1}{2\nu}
\left(
\frac{1-\epsilon}{\nicefrac{1}{2}+\epsilon}
\right)^{\frac{1}{2}}
\norm{
\nabla \delta^{k+1} p^n
}^2
+
\frac{\nu}{2}
\left(
\frac{\nicefrac{1}{2}+\epsilon}{1-\epsilon}
\right)^{\frac{1}{2}}
\norm{
\Delta \delta^{k+1} \bu^{n+1}
}^2
,
\end{align*}
summing over the entire equation from $n=1$ to $m-1$ and plugging in the estimate of Eq.~\eqref{eq:est14}.
\end{proof}
In terms of Lemma~\ref{lem:pressure-stability-estimate}, we can readily show the following stability estimate:
\begin{lemma}[Stability estimate  for the Stokes equation]
    \label{lem:stability-stokes}
    Given Assumptions~\ref{ass:domain}, \ref{ass:constants}, \ref{ass:solution}, and \ref{ass:bootstrapping}, we have
    \begin{align*}
        \norm{\nabla \bu^m}^2
        +  \tau\sum_{n=0}^{m} \norm{\Delta \bu^n}^2
        \leq C_\hu^2
    \end{align*}
    where $C_\hu$ depends on $M$ and $T$, but not on $\tau$.
\end{lemma}
\begin{proof}
    We test Eq.~\eqref{eq:scheme-u} with $(-\Delta) \delta^{k+1} \bu^{n+1}$ and obtain
    \begin{align*}
         &
        \inner{ \tD^k \bu^{n+1} ,  (-\Delta) \delta^{k+1} \bu^{n+1} }
        \!+\! 2 \tau \nu  \inner{\Delta \delta^k \bu^{n+1}, \Delta \delta^{k+1} \bu^{n+1}}
        \!+\! 2 \tau \inner{(\delta^{k+1} \hu^n \cdot \nabla) \delta^{k+1} \hu^n, (-\Delta) \delta^{k+1} \bu^{n+1}}
        \\&\quad
        + 2 \tau \inner{\nabla \delta^{k+1} p^n, (-\Delta) \delta^{k+1} \bu^{n+1} }
        = 2 \tau \inner{\hf(t^{n+k}, \delta^{k+1} \theta^n), (-\Delta) \delta^{k+1} \bu^{n+1} }
        .
    \end{align*}
    We start estimating the advection term by Eq.~\eqref{eq:estimate-advection-term-u} of Lemma~\ref{lem:estimate-advection-term}, i.e.
    \begin{align}
    (I) :=
    | \langle(\delta^{k+1} \hu^n\cdot \nabla) \delta^{k+1} \hu^n, &(-\Delta) \delta^{k+1} \bu^{n+1}\rangle |\,
        \leq 
        \epsilon
        \norm{\Delta \delta^{k+1} \bu^{n+1}}^2
        +
        \epsilon
        \norm{\Delta \delta^{k+1} \hu^{n}}^2
        +
        C
        \norm{\nabla \delta^{k+1} \hu^n}^4,
        \nonumber
    \end{align}
    
            where we used $\norm{\delta^{k+1} \hu^n} \leq C$ by Lemma~\ref{lem:weak-stability}.
            The second term can be trivially simplified to $\norm{ \Delta \delta^{k+1} \hu^n }^2 \leq C ( \norm{\Delta \bu^n}^2 + \norm{\Delta \bu^{n-1}}^2) $.
            Similarly, for the last term, we get
            $\norm{ \nabla \delta^{k+1} \hu^n }^4 \leq C ( 
            \norm{\nabla \hu^n}^4
            + 
            \norm{\nabla \hu^{n-1}}^4
            ) $.
            To apply Lemma~\ref{lem:Gronwall-reformulated}, we split the quartic terms into $\norm{\nabla \hu^{i}}^2$ taking the role of the $d_n$, since weighted by $\tau$, it is summable by Lemma~\ref{lem:weak-stability}, and into $\norm{\nabla \bu^{i}}^2$. Thus, using the boundedness of $|\eta^n|$, we estimate
            $\norm{ \nabla \delta^{k+1} \hu^n }^4 \leq C ( 
            \norm{\nabla \hu^n}^2
            \norm{\nabla \bu^n}^2 
            + 
            \norm{\nabla \hu^{n-1}}^2
            \norm{\nabla \bu^{n-1}}^2
            ) $.
            Hence, after redefining $\epsilon$, we obtain
        
    \begin{align}
        (I)
        &
        \leq 
        \epsilon
        \norm{\Delta \delta^{k+1} \bu^{n+1}}^2
        \!
        +
        \!
        \epsilon
        \sum_{i=0}^1
        \norm{\Delta \bu^{n-i}}^2
        \!
        +
        \!
        C
        \sum_{i=0}^1
        \norm{\nabla \hu^{n-i}}^2
        \norm{\nabla \bu^{n-i}}^2
        .
        \label{eq:iib-stability-stokes-advection}
    \end{align}
    Estimating the right-hand side term from the heat equation yields
    \begin{align}
        \abs{
            \inner{\hf(\delta^{k+1} \theta^n),
                (-\Delta) \delta^{k+1} \bu^{n+1} }
        }
         &
        \leq
        \epsilon
        \norm{\Delta \delta^{k+1} \bu^{n+1} }^2
        +
        C
        \norm{\hf(\delta^{k+1} \theta^n)}^2
        \nonumber
        \\
         &
        \leq
        \epsilon
        \norm{\Delta \delta^{k+1} \bu^{n+1} }^2
        +
        C
        (C^2_{\hf_1} + \alpha^2 C^2_\theta)
        .
        \label{eq:iib-proof-stability-stokes-rhs}
    \end{align}
    Summing from $n=1$ to $m-1$,
    applying Lemma~\ref{lem:abb-equation-2},
    Lemma~\ref{lem:pressure-stability-estimate},
    and the estimates in Eqs.~\eqref{eq:iib-stability-stokes-advection}, and \eqref{eq:iib-proof-stability-stokes-rhs},
    in combination with omitting some positive terms and simplifying constants, yield
    \begin{align*}
        a_k
        \norm{\nabla \bu^{m}}^2
        &
        +
        2 \tau
        \hat{\epsilon}_S
        \sum_{n=1}^{m-1}
        \norm{\Delta \delta^{k+1} \bu^{n+1} }^2
        +
        2 \tau (\hat{f}_k - 2 \epsilon)
        \sum_{n=1}^{m-1}
        \norm{\Delta \bu^{n+1}}^2
        \leq
        a_k\norm{\nabla\bu^{1}}^2
        + 2\tau \hat{d}_k \norm{\Delta \bu^1}^2
        \\
        &
        + \norm{b_k \nabla \bu^{1} - c_k \nabla \bu^{0}}^2
        + 2\tau C
        \sum_{n=0}^{m-1}
        (1 + \norm{\nabla \hu^{n}}^2) \norm{\nabla \bu^{n}}^2
        + 2\tau C
        \sum_{n=0}^{m-1}
        \norm{\nabla \delta^{k+1} \bu^n}^2
        +
        C
        ,
    \end{align*}
    where we defined
    \[
        \hat{\epsilon}_S :=
        \hat{a}
        - \sqrt{\frac{\nicefrac{1}{2}+\epsilon}{1-\epsilon}}
        - 3\epsilon
        .
    \]
    Since $\lim_{\epsilon\to0}\hat{\epsilon}_S = \hat{a} -\nicefrac{1}{\sqrt{2}} \geq 0.712 - 0.707 > 0$ for $k\geq 3>\frac{3}{2}+\sqrt{2}$, the lower limit for $\sqrt{2}\hat{a}>1$, we can find some $\epsilon > 0$ small enough to obtain a positive coefficient in front of $\norm{\Delta \bu^{n+1}}^2$.

    Using Lemma~\ref{lem:Gronwall-reformulated} with $  2 \tau C \sum_{n=0}^{m-1} (1 + \norm{\nabla \hu^n}^2) \leq 2 C (T + M) =: M_a$, $M_b = M_c = 0$ and $ M_d = 2 T C$, we obtain the claim.
\end{proof}

\begin{remark}We remark that Lemma~\ref{lem:stability-stokes} in combination with Lemma~\ref{lem:weak-stability} directly implies the stability of
\[
\norm{\nabla \hu^m}^2 + \tau\sum_{n=0}^{m} \norm{\Delta \hu^n}^2 \leq M C_\hu
\quad
\textmd{ and }
\quad
\norm{\nabla \delta^{k+1} \hu^m}^2 + \tau\sum_{n=0}^{m} \norm{\Delta \delta^{k+1} \hu^n}^2 \leq C C_\hu
.
\]
\end{remark}

\subsection{Error estimates}
We now turn towards the error estimate and introduce the error evolution system. 
The error for the heat equation is given by
\begin{align}
    \tD^l e_\theta^{n+1}
    +               & 2 \tau R_\star^{n+1}
    - 2 \tau \kappa \Delta \delta^{l} e_\theta^{n+1}
    =
    -
    \Xi_{n+1}(\theta)
    \label{eq:error-T}
    \intertext{and for the Stokes equation by}
    \tD^k \be^{n+1}_{\hu} &- 2 \tau \nu \Delta \delta^k \be^{n+1}_{\hu}
                     + 2 \tau \nabla \delta^{k+1} e^n_{p}
    + 2 \tau S_\star^{n+1}
    + 2 \tau S_1^{n+1}
    + 2 \tau S_2^{n+1}
    =
    0
    ,
    \label{eq:error-u}
    \intertext{with the right-hand side terms for the Stokes equation}
    S_1^{n+1} =     &\,
    (\delta^{k+1} \hu^n \cdot \nabla) \delta^{k+1} \hu^{n}
    - (\hu^{n+k}_\star \cdot \nabla) \hu^{n+k}_\star,
        \label{eq:error-S1}
    \\
    S_2^{n+1} =     &
    \left(
    \nabla \delta^{k+1} p^n_\star - \nabla p^{n+k}_\star
    \right)
    +
    \frac{1}{2\tau} ( \tD^k \hu^{n+1}_\star - 2 \tau \dot{\hu}^{n+k}_\star )
    - \Delta
    \left(
    \delta^k \hu^{n+1}_\star - \hu^{n+k}_\star
    \right)
    ,
        \label{eq:error-S2}
    \intertext{and the coupling terms}
    R_\star^{n+1} &= 
    \delta^{l+1} \hu^{n} \cdot \nabla \delta^{l+1} \theta^{n}
    -
    \hu^{n+l}_\star \cdot \nabla \theta^{n+l}_\star
    \nonumber
    \\
    & =               
    \delta^{l+1} \he^n_\hu \cdot \nabla \delta^{l+1} \theta^n
    + \delta^{l+1} \hu^n_\star \cdot \nabla \delta^{l+1} e_\theta^n
    \nonumber
    \\&\qquad
    - (\hu^{n+l}_\star - \delta^{l+1} \hu^n_\star) \cdot \nabla \theta^{n+l}_\star
    - \delta^{l+1} \hu^n_\star \cdot \nabla ( \theta^{n+l}_\star - \delta^{l+1} \theta^{n}_\star),
    \quad&\textmd{and}
    \label{eq:error-Rstar}
    \\
        S_\star^{n+1} &=\, 
    \hf_2(\theta^{n+k}_\star) - \hf_2(\delta^{k+1} \theta^{n}).
    \label{eq:error-Sstar}
\end{align}

For the error estimates, we have to show that the scaling with $\eta^n$ does not negatively impact the asymptotical convergence of our scheme.
For this, the following identity will be heavily used, connecting the error of $\he^n$ to $\be^n$:
\begin{corollary}\label{cor:e-ebar}
    Assume there exists a bound $\abs{1-\eta^n} \leq C_0 \tau^2$, and $\norm{\nabla \bu^n}^2 \leq C_\hu^2$ holds. Then, for all $k \geq 0$, we have
    \begin{align}
        \norm{\nabla\he^n_\hu}^2
         & \leq
        2
        \norm{\nabla\be^n_\hu}^2
        +
        2 C_0^2 C_\hu^2 \tau^4
        ,
        \quad
        \textmd{ and }
        \quad
        \norm{\nabla\delta^k\he^n_\hu}^2
         \leq
        2
        \norm{\nabla\delta^k\be^n_\hu}^2
        +
        2
        C
        C_0^2 C_\hu^2 \tau^4
        .
        \label{eq:e-bar}
    \end{align}
\end{corollary}
\begin{proof} Follows directly from $
        \norm{\nabla\he^n_\hu}
        \leq  \norm{\nabla\be^n_\hu} + \norm{\nabla(\hu^n - \bu^n)}
        \leq  \norm{\nabla\be^n_\hu} + \abs{1-\eta^n}\norm{\nabla\bu^n}
    $
    .
\end{proof}
We will later show that we can always find an appropriate constant $C_0$ for small enough time step sizes $\tau\leq\tau_0$.
But first, we will use it as an assumption to show the following key estimate that will be needed several times; firstly, to estimate $\norm{S^{n+1}_1}^2$, and secondly, to estimate $\delta^{k+1}p^n$.
\begin{lemma}\label{lem:unablu-estimate}
    In addition to Assumptions~\ref{ass:domain}, \ref{ass:constants}, \ref{ass:solution}, and \ref{ass:bootstrapping}, we suppose that there are bounds $C_0, c_\xi > 0$ such that for all $0\leq m\leq N-1$ we have $\abs{1 - \eta^n} \leq \tau^2 C_0,$ and $\xi^n \geq c_\xi$.
    Then, the estimate
    \begin{align*}
        \lVert
        (\hu^{n+k}_\star \cdot \nabla) \hu^{n+k}_\star
        -
        (\delta^{k+1} \hu^n \cdot  \nabla) \delta^{k+1} \hu^n
        \rVert^2
         &
        \leq
        C
        (1+\norm{\Delta \delta^{k+1} {\hu}^{n} }^2)
        (
        \norm{ \nabla \delta^{k+1} \be^n_\hu }^2
        + \tau^4 C_\hu^2 C_0^2
        )
        \\&\qquad
        +
        C
        \norm{
            \zeta_{n,k+1}( \nabla \hu )
        }^2
    \end{align*}
    holds for a constant $C > 0 $, independent of $\tau$ and $C_0$.
\end{lemma}
\begin{proof}
    Adding and subtracting the term $(\hu^{n+k}_\star \cdot \nabla) \hu^{n+k}_\star$ and applying the triangle inequality give
    \begin{align*}
        \lVert
        (\hu^{n+k}_\star \cdot \nabla) \hu^{n+k}_\star
        -
         &
        (\delta^{k+1} \hu^n \cdot  \nabla) \delta^{k+1} \hu^n
        \rVert^2
        \leq
        \norm{ (\delta^{k+1} \hu^n_\star \cdot \nabla) \delta^{k+1} \hu^{n}_\star - (\hu^{n+k}_\star\cdot \nabla) \hu^{n+k}_\star }
        \\&\quad
        +
        \lVert
        (\delta^{k+1} \hu^n_\star \cdot  \nabla) \delta^{k+1} \hu^n_\star
        -
        (\delta^{k+1} \hu^n \cdot  \nabla) \delta^{k+1} \hu^n
        \rVert
        =: (I) + (II)
        .
    \end{align*}
    For the first term on the right, we have
    \begin{align}
        (I)
                \nonumber
         &
        =
        \norm{
            \big((\delta^{k+1} \hu^n_\star - \hu^{n+k}_\star)\cdot
            \nabla\big)
            \delta^{k+1}
            \hu^n_\star
            +
            (\hu^{n+k}_\star \cdot \nabla) (\delta^{k+1}\hu^n_\star - \hu^{n+k}_\star)
        }
        \nonumber
        \\
         &
        \leq
        \norm{
            \zeta_{n,k+1}( \hu )
            \cdot
            \nabla
            \delta^{k+1}
            \hu^n_\star
        }
        +
        \norm{
            \hu^{n+k}_\star
            \cdot
            \nabla  \zeta_{n,k+1}( \hu )
        }
        \nonumber
                \\
         &
        \lesssim
        \norm{
            \zeta_{n,k+1}( \nabla \hu )
        }
        \norm{
            \delta^{k+1}
            \hu^n_\star
        }_{2}
        +
        \norm{
            \hu^{n+k}_\star
        }_2
        \norm{
            \zeta_{n,k+1}( \nabla\hu )
        }
                        \lesssim
        \norm{
            \zeta_{n,k+1}( \nabla \hu )
        }
        \nonumber
        ,
    \end{align}
    where we used that both,
    $\norm{\delta^{k+1}\hu^n_\star}_{2}$
    and
    $\norm{\hu^{n+k}_\star}_{2}$,
    are bounded by our assumptions on the solution.
    For the second term, we have
    \begin{align*}
        (II) &= 
        \norm{
            (\delta^{k+1} \he^n_\hu \cdot \nabla) \delta^{k+1} \hu^{n}
            +
            (\delta^{k+1} \hu^n_\star \cdot\nabla) \delta^{k+1} \he^{n}_\hu
        }
        ,
        \\ & \lesssim
        \norm{ (\delta^{k+1} \he^n_\hu \cdot \nabla) \delta^{k+1} \hu^{n} }
        + \norm{ (\delta^{k+1} \hu^n_\star  \cdot\nabla) \delta^{k+1} \he^{n}_\hu }
        \\ & \lesssim
        \norm{ \delta^{k+1} \he^n_\hu }_{L^4(\Omega)}
        \norm{\nabla \delta^{k+1} \hu^{n} }_{L^4(\Omega)}
        +
        \norm{ \delta^{k+1} \hu^n_\star  }_{L^\infty(\Omega)}
        \norm{ \nabla \delta^{k+1} \he^{n}_\hu }.
                \intertext{Using the embeddings of $H^1(\Omega) \hookrightarrow L^4(\Omega)$, $H^2(\Omega) \hookrightarrow L^\infty(\Omega)$, and the assumed $H^2$-regularity give
        }
                              (II)   & \lesssim
        \norm{ \nabla \delta^{k+1} \he^n_\hu }
        \norm{\Delta \delta^{k+1} \hu^{n} }
        +
        \norm{ \delta^{k+1} \hu^n_\star  }_{2}
        \norm{ \nabla \delta^{k+1} \he^{n}_\hu }.
                \intertext{Again, using
            our boundedness conditions on the solution and Corollary~\ref{cor:e-ebar}, yield}
                              (II)   & \lesssim
        (1+\norm{\Delta \delta^{k+1} \hu^{n} })
        \norm{ \nabla \delta^{k+1} \he^n_\hu }
                                 \lesssim
        (1+\norm{\Delta \delta^{k+1} {\hu}^{n} })
        ( \norm{ \nabla \delta^{k+1} \be^n_\hu } + \tau^2 C_0 C_\hu )
        .
    \end{align*}
    Collecting both estimates for $(I)$ and $(II)$, squaring both sides, and redefining constants yield the result.
\end{proof}

\begin{lemma}\label{lem:error-estimate-pressure}
Given Assumptions~\ref{ass:domain}, \ref{ass:constants}, \ref{ass:solution}, and \ref{ass:bootstrapping}, and constants
    $C_0, c_\xi > 0$ as well as an $m\in\Ii$, such that for all $n\leq m-1$
    we have $\abs{1 - \eta^n} \leq \tau^2 C_0$, and $\xi^n \geq c_\xi$.
    Then, there is a $C\geq0$ independent of $C_0$, $\tau$ and $m$ such that
    \begin{align*}
    \sum_{n=1}^{m-1}
    \inner{
    \nabla \delta^{k+1} e_p^n
    ,
    \Delta \delta^{k+1} \be_\hu^{n+1}
    }
    &
    \leq
    \nu
    \sqrt{
    \frac{\nicefrac{1}{2}+\epsilon}{1-\epsilon}
    }
    \sum_{n=1}^{m-1}
    \norm{
    \Delta \delta^{k+1} \be_\hu^{n+1}
    }^2
    + C \sum_{n=1}^{m-1} \norm{\delta^{k+1}e_\theta^n}
    + C \tau^3
    \\&
    \qquad
    +
    C 
    \sum_{n=1}^{m-1}
    (1+\norm{\Delta \delta^{k+1} {\hu}^{n} }^2) ( \norm{ \nabla \delta^{k+1} \be^n_\hu }^2 + \tau^4 C_0^2 C_\hu^2
    )
    .
    \end{align*}
\end{lemma}
\begin{proof}
We start with plugging the solution into Eq.~\eqref{eq:gen:pressure}, and applying $\delta^{k}$, to obtain for all $n\geq2$:
\begin{align*}
\inner{
\nabla \delta^{k} p^{n+1}_\star,
\nabla q
}
=
\inner{ \nu(\Delta - \nabla \nabla \cdot) \nabla \delta^{k} \hu^{n+1}_\star
- 
\delta^{k}((\hu^{n+1}_\star \cdot \nabla) \hu^{n+1}_\star)
+ \delta^{k}(\hf(t^{n+1}, \theta^{n+1}_\star))
, \nabla q}
.
\end{align*}
Subtracting this equation from Eq.~\eqref{eq:pressure-before-recursion} multiplied by $k$ and using the definition of $\hh^n$, we find
\begin{align*}
\inner{ \nabla \delta^{k} e_p^{n+1}, \nabla q }
&=
\inner{ \nu(\Delta - \nabla \nabla \cdot) \delta^{k} \be_\hu^{n+1} 
- \he_\hh^n
, \nabla q}
,
\end{align*}
with
\begin{align*}
\he_\hh^n
=
\hf(t^{n+k}, \delta^{k+1} \theta^n)
- \delta^{k}(\hf(t^{n+1}, \theta^{n+1}_\star))
+  \delta^{k}((\hu^{n+1}_\star \cdot \nabla) \hu^{n+1}_\star)
- (\delta^{k+1} \hu^n \cdot \nabla) \delta^{k+1} \hu^n
.
\end{align*}
Thus, with the same arguments as in the derivation of Eq.~\eqref{eq:pressure-est15} in the proof of Lemma~\ref{lem:pressure-stability-estimate}, we obtain 
\begin{align*}
\sum_{n=1}^{m-1}
\norm{\nabla \delta^{k+1}e_p^{n+1}
}^2
&
\leq
\nu^2
\frac{\nicefrac{1}{2}+\epsilon}{1-\epsilon}
\sum_{n=1}^{m-1}
\norm{
\Delta \delta^{k+1} \be_\hu^{n+1}
}^2
+
C
\sum_{n=1}^{m-1}
\norm{
\nabla \delta^{k+1} \be_\hu^{n+1}
}^2
\\&
+
C
\norm{
\nabla 
e_p^1
- \nabla p_s(\nu \be_\hu^{1})}^2
+ 
C
\sum_{i=1}^{m-1}
\norm{
\delta^{k+1}
\he_\hh^i
}^2
.
\end{align*}
We split up $\he_\hh^n = (I) + (II) + (III) + (IV)$ into
\begin{align*}
(I) & :=
\hf_1(t^{n+k})-\delta^k(\hf_1(t^{n+1}))
,\\ (II) & :=
\hf_2(\delta^{k+1} \theta^n)
- \delta^{k}(\hf_2(\theta^{n+1}_\star))
,\\ 
(III) & :=
(\hu^{n+k}_\star \cdot \nabla) \hu^{n+k}_\star
- (\delta^{k+1} \hu^n \cdot \nabla) \delta^{k+1} \hu^n
,\textmd{ and }\\ 
(IV) & := \delta^{k}((\hu^{n+1}_\star \cdot \nabla) \hu^{n+1}_\star) - (\hu^{n+k}_\star \cdot \nabla) \hu^{n+k}_\star.
\end{align*}

We obtain by Lemma~\ref{lem:time-derivative-and-extrapolation} $\norm{(I)} \leq \norm{\zeta_{n+1,k}(\hf_1)}$, and
\begin{align*}
\norm{
(II)
}
&\leq 
\norm{ \hf_2(\delta^{k+1} \theta^n) - \hf_2(\delta^{k+1} \theta^n_\star) }  + \norm{  \hf_2(\delta^{k+1} \theta^n_\star)  -\hf_2(\theta^{n+k}_\star)  } 
+ \norm{ \hf_2(\theta^{n+k}_\star)  - \delta^{k}(\hf_2(\theta^{n+1}_\star)) }
\\
&\leq
\alpha \norm{\delta^{k+1}e_\theta^n}
+
\alpha \norm{\zeta_{n,k+1}(\theta)}
+
\norm{\zeta_{n+1,k}(\hf_2\circ \theta)}
.
\end{align*}
By Lemma~\ref{lem:unablu-estimate} we have
\begin{align*}
    \norm{(III)
}^2
     & \leq
    C
    (1+\norm{\Delta \delta^{k+1} {\hu}^{n} }^2)
    (
    \norm{ \nabla \delta^{k+1} \be^n_\hu }^2
    + \tau^4 C_0^2 C_\hu^2
    )
        +
    C
    \norm{
        \zeta_{n,k+1}( \nabla \hu )
    }^2
    ,
\end{align*}
and $(IV) =  \zeta_{n+1,k}(\hu \cdot \nabla \hu)$.
To apply Lemma~\ref{lem:time-derivative-and-extrapolation}, it suffices that
$\partial_t^2 (\hu \cdot \nabla \hu )\in L^2(\Omega)$.
The chain rule gives
\begin{align*}
    \partial_t^2 (\hu \cdot \nabla \hu )
    &=
    \partial_t
    (
    \dot\hu \cdot \nabla \hu
    +
    \hu \cdot \nabla \dot\hu
    )
    =
    \ddot\hu \cdot \nabla \hu
    +
    \hu \cdot \nabla \ddot\hu
    +
    2
    \dot\hu \cdot \nabla \dot\hu
    ,
    \intertext{and since}
    \norm{
        \partial_t^2 (\hu \cdot \nabla \hu )
    }
     & \leq
    \norm{\ddot\hu \cdot \nabla \hu}
    +
    \norm{
        \hu \cdot \nabla \ddot\hu
    }
    +
    2
    \norm{
        \dot\hu \cdot \nabla \dot\hu
    }
    \\&\lesssim
    \norm{\ddot\hu}_{1} \norm{\hu}_2
    +
    \norm{ \hu }_2 \norm{\ddot\hu }_1
    +
    \norm{
        \dot\hu
    }_1
    \norm{
        \dot\hu
    }_2
\end{align*}
is bounded by our assumptions on $\hu$, we conclude $\zeta_{n+1,k}(\hu \cdot \nabla \hu) \in L^2(\Omega)$.
Thus, 
\begin{align}
\sum_{n=1}^{m-1}
\norm{\nabla \delta^{k+1}e_p^{n+1}
}^2
&
\leq
\nu^2
\frac{\nicefrac{1}{2}+\epsilon}{1-\epsilon}
\sum_{n=1}^{m-1}
\norm{
\Delta \delta^{k+1} \be_\hu^{n+1}
}^2
+ C \sum_{n=1}^{m-1} \norm{\delta^{k+1}e_\theta^n}^2
+ C \tau^3
\nonumber
\\&
\qquad
+
C 
\sum_{n=1}^{m-1}
(1+\norm{\Delta \delta^{k+1} {\hu}^{n} }^2) ( \norm{ \nabla \delta^{k+1} \be^n_\hu }^2 + \tau^4 C_0^2 C_\hu^2
)
,
\label{lem:pressure-estimate}
\end{align}
where we used Lemma~\ref{lem:time-derivative-and-extrapolation} to bound the terms with $\zeta$ by $C \tau^3$.
The conclusion follows from
\begin{align*}
\inner{
\nabla \delta^{k+1} e_p^n
,
\Delta \delta^{k+1} \be_\hu^{n+1}
}
\leq
\frac{1}{2\nu}
\left(
\frac{1-\epsilon}{\nicefrac{1}{2}+\epsilon}
\right)^{\frac{1}{2}}
\norm{
\nabla \delta^{k+1} e_p^n
}^2
+
\frac{\nu}{2}
\left(
\frac{\nicefrac{1}{2}+\epsilon}{1-\epsilon}
\right)^{\frac{1}{2}}
\norm{
\Delta \delta^{k+1} \be_\hu^{n+1}
}^2
.
\end{align*}
\end{proof}
We will now show the following error estimate:
\begin{lemma}[Error estimate]
    \label{lem:error-boussinesq}
    Given Assumptions~\ref{ass:domain}, \ref{ass:constants}, \ref{ass:solution}, and \ref{ass:bootstrapping}, we further assume the bounds $c_\xi, c_\eta, C_0 > 0$, such that,
    \[
        \xi^{n} \geq c_\xi,\, \eta^{n} \geq c_\eta,
        \textmd{ and }
        \abs{1 - \eta^n} \leq C_0 \tau^2 
        ,
        \textmd{ for all }
        0 \leq n \leq m-1
        ,
    \]
    where $0\leq m \leq N$.
    Then, we have
    \[
        \norm{\nabla \be^m_\hu}^2
        + \norm{\nabla {e}_\theta^m}^2
        + \tau \sum_{n=1}^{m-1} \norm{\Delta \be_\hu^{n+1}}^2
        + \tau \sum_{n=1}^{m-1} \norm{\Delta e_\theta^{n+1}}^2
        \leq C_e^2 \tau^4
        ,
    \]
    with $C_e^2 = \tilde{C}_e^2 (1 + C_0^2)$, where $\tilde{C}_e$ is independent of $C_0$ and $\tau$.
\end{lemma}
\begin{proof}
    We divide the proof into two steps:
    In Step 1, we will derive an estimate in which the heat equation is nearly in the correct shape for applying the Gr\"onwall Lemma~\ref{lem:Gronwall-reformulated}. However, the direct use is prohibited by the presence of a velocity term.
    In Step 2, we will repeat the same for the Stokes equations. Here, temperature terms will prevent the application of the Gr\"onwall Lemma.
    Finally, in Step 3, we will add both terms, apply Gr\"onwall, and get the desired bounds.

        {\bf Step 1: The heat equation:}
    We test the error equation~\eqref{eq:error-T} of the heat equation with $(-\Delta) \delta^{l+1} e_\theta^{n+1}$ and obtain
    \begin{align}
        &
        \inner{\nabla \tD^l e_\theta^{n+1}\!\!, \nabla \delta^{l+1} e_\theta^{n+1}}
        \! -  \!
        2 \tau \inner{R^{n+1}_\star\!\!, \Delta \delta^{l+1} e_\theta^{n+1}}
        \! +  \!
        2 \tau \kappa \inner{\Delta \delta^{l} e_\theta^{n+1}\!\!, \Delta \delta^{l+1} e_\theta^{n+1}}
        \nonumber
        \\&\quad
        =
        \inner{\Xi_{n+1}(\theta), \Delta \delta^{l+1} e_\theta^{n+1}} .
        \label{eq:iib-tested-error-heat}
    \end{align}
    To estimate the second term, we first apply Young's inequality:
    \begin{align*}
        \abs{\inner{R^{n+1}_\star, (-\Delta) \delta^{l+1} e_\theta^{n+1}}}
        \leq
        \kappa \epsilon \norm{ \Delta \delta^{l+1} e_\theta^{n+1} }^2
        +
        \frac{1}{4\kappa \epsilon }
        \norm{R^{n+1}_\star}^2.
    \end{align*}
    Next, applying the triangle inequality to $R^{n+1}_\star$ from Eq.~\eqref{eq:error-Rstar} gives
    \begin{align*}
        \norm{ R^{n+1}_\star }
        &
        \leq     
        \norm{\delta^{l+1} \he^n_\hu \cdot \nabla \delta^{l+1} \theta^n}
        + \norm{\delta^{l+1} \hu^n_\star \cdot \nabla \delta^{l+1} e_\theta^n}
        + \norm{\zeta_{n,l+1}(\hu)\cdot \nabla \theta^{n+l}_\star}
        + \norm{\delta^{l+1} \hu^n_\star \cdot \zeta_{n,l+1}(\nabla \theta)}
        \\
        &
        \lesssim 
        \norm{\nabla \delta^{l+1} \he^n_\hu}
        \norm{ \Delta \delta^{l+1} \theta^n}
        +
        \norm{\nabla \delta^{l+1} \hu^n_\star}
        \norm{ \nabla \delta^{l+1} e_\theta^n}^{1/2}
        \norm{ \Delta \delta^{l+1} e_\theta^n}^{1/2}
        \\ & \quad
        + \norm{\zeta_{n,l+1}(\nabla \hu)} \norm{ \Delta \theta^{n+l}_\star}
        + \norm{\delta^{l+1} \hu^n_\star}_2 \norm{\zeta_{n,l+1}(\nabla \theta)}
        ,
    \end{align*}
    where we used Cauchy-Schwarz in combination with $H^1(\Omega)\hookrightarrow L^4(\Omega) $, Lemma~\ref{lem:vectorial-ladyzhenskaya-inequality} and $H^2(\Omega)\hookrightarrow C(\Omega)$ in our estimates.
    Assumption~\ref{ass:solution} guarantees $\norm{\nabla \delta^{l+1} \hu^n_\star}, \norm{\Delta \theta^{n+l}_\star}, \norm{\delta^{l+1} \hu^n_\star}_2 \lesssim 1$ and thus
    \begin{align*}
        \norm{ R^{n+1}_\star }^2
        &
        \lesssim 
        \norm{\nabla \delta^{l+1} \he^n_\hu}^2
        \norm{ \Delta \delta^{l+1} \theta^n}^2
        +
        \norm{ \nabla \delta^{l+1} e_\theta^n}
        \norm{ \Delta \delta^{l+1} e_\theta^n}
                + \norm{\zeta_{n,l+1}(\nabla \hu)}^2 + \norm{\zeta_{n,l+1}(\nabla \theta)}^2.
        \intertext{
            Setting $k=l+1$ in Corollary~\ref{cor:e-ebar} and multiplying \eqref{eq:e-bar} with $\norm{\Delta\delta^{l+1}\theta^n}^2$ yield in combination with Young's inequality the bound
        }
        \norm{ R^{n+1}_\star }^2
        &
        \leq
        C
        \norm{\nabla \delta^{l+1} \be^n_\hu}^2
        \norm{ \Delta \delta^{l+1} \theta^n}^2
        +
        C
        C_0^2
        C_\hu^2
        \tau^4
        \norm{\Delta \delta^{l+1}\theta^n}^2
        +
        \frac{C}{\kappa^2 \epsilon^2}
        \norm{ \nabla \delta^{l+1} e_\theta^n}^2
        \\&\quad
        + 4 \kappa^2 \epsilon^2 \norm{ \Delta \delta^{l+1} e_\theta^n}^2
        + C (\norm{\zeta_{n,l+1}(\nabla \hu)}^2 + \norm{\zeta_{n,l+1}(\nabla \theta)}^2)
        .
    \end{align*}
    For the fourth term in \eqref{eq:iib-tested-error-heat}, we have
    \begin{align*}
        \abs{\inner{\Xi_{n+1}(\theta),  \Delta \delta^{l+1} e_\theta^{n+1}}}
        \leq
        2 \tau \epsilon \kappa \norm{ \Delta \delta^{l+1} e_\theta^{n+1}}^2
        + \frac{1}{8\tau \epsilon\kappa} \norm{\Xi_{n+1}(\theta)}^2
        ,
    \end{align*}
    where $\epsilon > 0$ arbitrary will be chosen later.
    Applying Eq.~\eqref{eq:factorizationD} and \eqref{eq:factorizationextr2} of Lemma~\ref{eq:factorizationextr2} to Eq.~\eqref{eq:iib-tested-error-heat}, and collecting all the previous estimates, yield
    \begin{align*}
         &
        a_l (
        \norm{\nabla e^{n+1}_\theta}^2
        -\norm{\nabla e^{n}_\theta}^2
        )
        +
        \norm{ b_l \nabla e^{n+1}_\theta - c_l \nabla e^{n}_\theta }^2
        -
        \norm{ b_l \nabla e^{n}_\theta - c_l \nabla e^{n-1}_\theta }^2
        + 2 \tau \kappa (\hat{a}_l - \epsilon) \norm{\Delta \delta^{l+1} e^{n+1}_\theta }^2
        \\
         &
        - 2 \tau \kappa \epsilon \norm{\Delta \delta^{l+1} e^{n}_\theta }^2
        + 2 \tau \kappa
        (
        (\hat{d}_l + \hat{f}_l) \norm{ \Delta e^{n+1}_\theta }^2
        -
        \hat{d}_l \norm{ \Delta e^{n}_\theta }^2
        )
        \leq
        \frac{2 \tau 2C}{\kappa \epsilon }
        \norm{\nabla \delta^{l+1} \be^n_\hu}^2
        \norm{ \Delta \delta^{l+1} \theta^n}^2
        \\ &
        \!+\!\frac{2\tau C}{\kappa^3 \epsilon^3} \norm{ \nabla \delta^{l+1} e_\theta^n}^2
        \!\!+\! \frac{2\tau C}{\kappa \epsilon }
        \norm{\zeta_{n,l+1}(\nabla \hu)}^2
        \!\!+\!\frac{\tau}{2\kappa \epsilon }
        \norm{\zeta_{n,l+1}(\nabla \theta)}^2
        \!\!+\! \frac{C}{\tau} \norm{\Xi_{n+1}(\theta)}^2
        \!\!+\! C C_0^2 C_\hu^2 \tau^5 \norm{\Delta \delta^{l+1}\theta^n}^2
        .
    \end{align*}
    Summing up from $n = 1$ to $m-1$,  using Lemma~\ref{lem:time-derivative-and-extrapolation},
    omitting some positive terms,
    using
    $a_l\norm{\nabla e^{1}_\theta}^2 \leq C \tau^4$,
    $\norm{ b_l \nabla e^{1}_\theta}^2 \leq C \tau^4$,
    and
    $\hat{d}_l \norm{ \Delta e^{1}_\theta }^2 \leq C \tau^3$,
    as well as the relation
    \begin{align*}
        C C_0^2 C_\hu^2 \tau^5 \sum_{n=1}^{m-1} \norm{\Delta \delta^{l+1}\theta^n}^2 \leq C C_0^2 C_\hu^2 C_\theta^2 \tau^4 \leq C C_0^2 \tau^4
    \end{align*}
    by Lemma~\ref{lem:stability-heat}, yields
    \begin{align}
         &
        a
        \norm{\nabla e^{m}_\theta}^2
                + 2 \tau \kappa
        \hat{\epsilon}_\theta
        \sum_{n=1}^{m-1}
        \norm{\Delta \delta^{l+1} e^{n+1}_\theta }^2
        + 2 \tau \kappa
        \hat{f}
        \sum_{n=1}^{m-1}
        \norm{ \Delta e^{n+1}_\theta }^2
                                                                                                \nonumber
        \\ & \leq
        \frac{2\tau C}{\kappa^3 \epsilon^3} \sum_{n=1}^{m-1} \norm{ \nabla  e_\theta^n}^2
        +
        \frac{2 \tau C}{\kappa \epsilon } \sum_{n=1}^{m-1} \norm{ \Delta \delta^{l+1} \theta^n}^2 \norm{\nabla \delta^{l+1} \be^n_\hu}^2
        +
        C(1+C_0^2) \tau^4
        ,
        \label{eq:error-result-T}
    \end{align}
    where we defined $ \hat{\epsilon}_\theta := \hat{a} - 2\epsilon $, which is positive for $\epsilon$ small enough, and estimated 
    $ \norm{ \nabla \delta^{l+1} e_\theta^n}^2\allowbreak\lesssim~\norm{ \nabla e_\theta^n}^2 \norm{ \nabla e_\theta^{n-1}}^2$.
    
        {\bf Step 2: The Stokes Equation:}
    We test Eq.~\eqref{eq:error-u} with $ (-\Delta)\delta^{k+1} \be^{n+1}_\hu$, and obtain
    \begin{align}
         &
        \inner{\nabla \tD^k \be^{n+1}_\hu, \nabla\delta^{k+1} \be^{n+1}_\hu} + 2 \tau \nu \inner{\Delta \delta^k \be^{n+1}_\hu, \Delta\delta^{k+1}  \be^{n+1}_\hu}
        + 2 \tau \inner{\nabla \delta^{k+1} e^n_p, (-\Delta)\delta^{k+1} \be^{n+1}_\hu}
        \label{eq:iib-tested-error-stokes}
        \\&\quad
        + 2\tau\inner{S^{n+1}_\star+S^{n+1}_1+S^{n+1}_2,(-\Delta) \delta^{k+1} \be^{n+1}_\hu}
        = 0
        .
        \nonumber
    \end{align}
    Firstly, we bound the coupling term $S^{n+1}_\star$ from Eq.~\eqref{eq:error-Sstar} and get
    \begin{align}
        \norm{
            S^{n+1}_\star
        }^2
        &
        =    
        \norm{
            \hf_2(\theta^{n+k}_\star) - \hf_2(\delta^{k+1} \theta^{n})
        }^2
        \leq
        2
        \norm{
            \hf_2(\theta^{n+k}_\star)
            -\hf_2(\delta^{k+1}\theta^n_\star)
        }^2
        +
        2
        \norm{
            \hf_2(\delta^{k+1}\theta^n_\star)
            - \hf_2(\delta^{k+1} \theta^{n})
        }^2.
        \nonumber
                \intertext{Using that $\hf_2$ is Lipschitz further yields}
        &
        \norm{
            S^{n+1}_\star
        }^2
        \leq 
        2
        \alpha^2
        \norm{
            \theta^{n+k}_\star
            -\delta^{k+1}\theta^n_\star
        }^2
        +
        2
        \alpha^2
        \norm{
            \delta^{k+1}e_\theta^n
        }^2
        \leq
        2
        \alpha^2
        \norm{
            \zeta_{n,k+1}(\theta)
        }^2
        +
        2
        \alpha^2
        \norm{
            \delta^{k+1}e_\theta^n
        }^2
        ,
        \label{eq:iib-Sstar}
    \end{align}
    and thus
    \begin{align}
        \abs{ \inner{S^{n+1}_\star, (-\Delta) \delta^{k+1} \be^{n+1}_\hu } }
        \leq
        \frac{
            \alpha^2
        }{
            2 \epsilon \nu
        }
        \norm{
            \zeta_{n,k+1}(\theta)
        }^2
        +
        \frac{
            \alpha^2
        }{
            2 \epsilon \nu
        }
        \norm{
            \delta^{k+1}e_\theta^n
        }^2
        +
        \epsilon \nu \norm{\Delta\delta^{k+1} \be_\hu^{n+1}}^2
        .
        \label{eq:error-res-Star}
    \end{align}
    Secondly, the term $S^{n+1}_1$ from Eq.~\eqref{eq:error-S1} can be bounded by
    \begin{align}
        | \langle S^{n+1}_1,& (-\Delta)\delta^{k+1} \be^{n+1}_\hu\rangle |
        \leq
        \nu \epsilon
        \norm{\Delta \delta^{k+1} \be^{n+1}_\hu}^2
        +
        C \norm{ S^{n+1}_1}^2
        .
        \nonumber
        \intertext{
            for all $\epsilon > 0$, with a now $\epsilon$-dependent $C$.
            Now, applying Lemma~\ref{lem:unablu-estimate} on the term $\norm{S^{n+1}_2}^2$ yields
        }
        &
        \leq
        \nu \epsilon \norm{\Delta \delta^{k+1} \be_\hu^{n+1} }^2
        +
        C
        \norm{
            \zeta_{n,k+1}( \nabla \hu )
        }^2
        +
        C
        (1+\norm{\Delta \delta^{k+1} {\hu}^{n} }^2)
        (
        \norm{ \nabla \delta^{k+1} \be^n_\hu }^2
        + \tau^4 C_0^2 C_\hu^2
        ).
                \label{eq:tpf-error-s1-s2}
    \end{align}

    Thirdly, the source term $S^{n+1}_2$ from Eq.~\eqref{eq:error-S2} gives directly from Young's inequality
    \begin{align}
        |
        \langle
        S_2^{n+1},
        &
        (-\Delta)\delta^{k+1}\be^{n+1}_\hu
        \rangle
        |
        \leq 
        \nu \epsilon
        \norm{\Delta\delta^{k+1}\be^{n+1}_\hu}^2
        +
        C
        \norm{S_2^{n+1}}^2
        \\
        & \leq 
        \nu \epsilon
        \norm{\Delta\delta^{k+1}\be^{n+1}_\hu}^2
        + C \norm{\zeta_{n, k+1}(\nabla p)}^2
        + \frac{C}{\tau^2} \norm{\Xi^{n+1}(\hu)}^2
        + C \norm{\zeta_{n+1,k}(\Delta \hu)}^2
        .
        \label{eq:tpf-error-s3}
    \end{align}

    Applying Lemma~\ref{lem:abb-equation-2} to Eq.~\eqref{eq:iib-tested-error-stokes}, collecting all the previous estimates in Eqs.~\eqref{eq:error-res-Star}, \eqref{eq:tpf-error-s1-s2}, and \eqref{eq:tpf-error-s3}, omitting some positive terms, simplifying constants, summing from $n=1$ to $m-1$, and finally applying Lemma~\ref{lem:error-estimate-pressure} yield
    \begin{align}
         &
        a_k \norm{\nabla \be_\hu^{m}}^2
                        +  2 \tau \nu   \sum_{n=1}^{m-1} \hat{f}_k \norm{  \Delta \be_\hu^{n+1}  }^2
                +2 \tau \nu
        \hat{\epsilon}_S
        \sum_{n=1}^{m-1}
        \norm{ \Delta \delta^{k+1} \be_\hu^{n+1} }^2
        \nonumber
        \\ &
        \leq
        2\tau C  \sum_{n=1}^{m-1} (1 + \norm{\Delta\delta^{k+1}{\hu}^n}^2 ) \norm{\nabla\delta^{k+1}\be^n_\hu}^2
        +
        2\tau C \sum_{n=1}^{m-1} \norm{\nabla e_\theta^n }^2
        +
        2 C (1+C_0^2) \tau^4
        ,\label{eq:error-result-u}
    \end{align}
    where we used 
    estimating $\norm{\delta^{k+1} e^n_\theta} \leq C \norm{ \nabla\delta^{k+1} e^n_\theta} \leq C {(\norm{\nabla e^n_\theta} + \norm{\nabla e^{n-1}_\theta})}$ by Poincar\'e, using Lemma~\ref{lem:time-derivative-and-extrapolation},
    as well as
    $a_k \norm{\nabla \be_\hu^{1}}^2 \leq C \tau^4$,
    $\norm{b_k \nabla \be_\hu^{1}}^2 \leq C \tau^4$,
    $ 2 \tau \nu \hat{d}_k \norm{  \Delta \be_\hu^{1}  }^2 \leq C \tau^4$,
    and defined the parameter
    \begin{align*}
    \hat\epsilon_S :=
    \hat{a}_k - 3 \epsilon - \sqrt{\frac{\nicefrac{1}{2} + \epsilon}{1 -\epsilon}}
    .
    \end{align*}
    Since for $k \geq 3$ we have $\hat{a}_k \geq 0.712 \geq \nicefrac{1}{\sqrt{2}}$, we can find $\epsilon > 0$, such that $\epsilon_S > 0$.

    {\bf Step 3: The fully-coupled System:}
    Adding up the estimates of the two previous systems
    in Eqs.~\eqref{eq:error-result-T} and \eqref{eq:error-result-u} yield
    \begin{align*}
         &
        a_{\min\{k,l\}} 
        (\norm{\nabla \be_\hu^{m}}^2 + \norm{\nabla e^{m}_\theta}^2)
        + 2 \tau \kappa \hat{\epsilon}_\theta\sum_{n=1}^{m-1} \norm{\delta^{k+1} \Delta e^{n+1}_\theta }^2
        +2 \tau \nu  \hat{\epsilon}_S \sum_{n=1}^{m-1} \norm{ \Delta \delta^{k+1} \be_\hu^{n+1} }^2
        \\ &
        +  2 \tau \nu  \hat{f}_k  \sum_{n=1}^{m-1} \norm{  \Delta \be_\hu^{n+1}  }^2
        + 2 \tau \kappa \hat{f}_l \sum_{n=1}^{m-1} \norm{ \Delta e^{n+1}_\theta }^2
        \leq
        \tau C \sum_{n=1}^{m-1} \norm{\nabla e_\theta^n }^2
        + C (1+C_0^2) \tau^4
        \\ &
        +
        \tau C  \sum_{n=1}^{m-1} (1 + \norm{\Delta\delta^{k+1}{\hu}^n}^2 ) \norm{\nabla\delta^{k+1}\be^n_\hu}^2
        +
        \tau C  \sum_{n=1}^{m-1}  \norm{ \Delta \delta^{l+1} \theta^n}^2 \norm{\nabla\delta^{l+1}\be^n_\hu}^2
        .
    \end{align*}
    By dividing by $a_{\min\{k,l\}}$ and using Lemma~\ref{lem:Gronwall-reformulated}, we obtain the claim
    with constants
    \begin{align*}
        \tau \sum_{n=1}^{m-1} c_n
         & =
        \tau
        \frac{C}{a_{\min\{k,l\}}}
        \sum_{n=1}^{m-1} \norm{\Delta \delta^{l+1}\theta^{n}}^2
        \leq
        C
        C_\theta^2
        =:  M_c
        ,
    \end{align*}
    \begin{align*}
        \tau \sum_{n=1}^{m-1} d_n
         & =
        \tau
        \frac{C}{a_{\min\{k,l\}}}
        \sum_{n=1}^{m-1} ( 1 + \norm{\Delta \delta^{k+1}{\hu}^{n}}^2)
        \leq
        C
        (T +
        C_\hu^2
        )
        =:  M_d
        ,
    \end{align*}
    and $M_a = 0$, $M_b = 2 C T/a_{\min\{k,l\}}$, which completes the proof.
\end{proof}

Next, we will prove first-order convergence for the auxiliary variable $r$, under the same conditions as in Lemma~\ref{lem:error-boussinesq}.
\begin{lemma}[Error for the auxiliary variable]
    \label{lem:error-aux}
    Given Assumptions~\ref{ass:domain}, \ref{ass:constants}, \ref{ass:solution}, and \ref{ass:bootstrapping}, and in addition the bounds $c_\eta, c_\xi, C_0 > 0$, such that
        $\xi^{n} \geq c_\xi,\, \eta^{n} \geq c_\eta$,
        and 
        $\abs{1 - \eta^n} \leq \tau^2 C_0$,
        for all
        $0 \leq n \leq m-1$ holds,
    where $0\leq m \leq N$, we have
                                                    \begin{align*}
        \abs{
            e^m_r
        }
        \leq
        \tau \, C_r (1 + \tau \sqrt{1 + C_0^2})
        ,
    \end{align*}
    where the constant $C_r > 0 $ depends on $c_\xi$, and $c_\eta$, but is independent of $\tau$ and $C_0$.
\end{lemma}
\begin{proof}
Subtracting $r^{m+1}_\star$ from $r^{m+1}$ in Eq.~\eqref{eq:weak-stability-explicit-r}, yields for the error
    \begin{align*}
    \abs{
    e^{m+1}_r
    }
    =
    \left|
    \exp\left(
    \tau
    \sum_{j=0}^m
        \frac{\nicefrac{d\mcE}{dt}(\theta^{j+1}, \bu^{j+1})}{\mcE(\theta^{j+1}, \bu^{j+1}) + \bar{C}}
    \right)
    - 
    \exp\left(
    \int_0^{t^{m+1}}
        \frac{\nicefrac{d\mcE}{dt}(\theta, \hu)}{\mcE(\theta, \hu) + \bar{C}}
    \textmd{d}t
    \right)
    \right|
    r^0.
    \end{align*}
    To estimate further, 
    we use the 
    elementary estimate
    $|e^x - e^y| \leq \max(e^x, e^y) |x-y|$
    and bound $\max(e^x, e^y)$ 
    with $e^{T/2}$ by using the  
    estimates in Eq.~\eqref{eq:weak-stability-estimate-f}, and \eqref{eq:weak-stability-estimate-g} also on 
    the terms containing
    $\theta^{j+1}_\star$, and
    $\hu^{j+1}_\star$.
    Further, we add and subtract the term $ \tau \sum_{j=0}^m \frac{\nicefrac{d\mcE}{dt}(\theta^{j+1}_\star, \hu^{j+1}_\star)}{\mcE(\theta^{j+1}_\star, \hu^{j+1}_\star) + \bar{C}} $ on the right-hand-side and get
    \begin{align*}
    \abs{
    e^{m+1}_r
    }
    \leq
    e^{T/2} r^0
    \left(
    \left|
    \tau
    \sum_{j=0}^m
        \frac{\nicefrac{d\mcE}{dt}(\theta^{j+1}, \bu^{j+1})}{\mcE(\theta^{j+1}, \bu^{j+1}) + \bar{C}}
        -
        \tau \sum_{j=0}^m \frac{\nicefrac{d\mcE}{dt}(\theta^{j+1}_\star, \hu^{j+1}_\star)}{\mcE(\theta^{j+1}_\star, \hu^{j+1}_\star) + \bar{C}}
    \right|
    \right.
    \\
    + 
    \left.
    \left|
        \tau \sum_{j=0}^m \frac{\nicefrac{d\mcE}{dt}(\theta^{j+1}_\star, \hu^{j+1}_\star)}{\mcE(\theta^{j+1}_\star, \hu^{j+1}_\star) + \bar{C}}
    -
    \int_0^{t^{m+1}}
        \frac{\nicefrac{d\mcE}{dt}(\theta, \hu)}{\mcE(\theta, \hu) + \bar{C}}
    \textmd{d}t
    \right|
    \right)
    .
    \end{align*}
    The second term is the error of a one-point quadrature rule. Since $\mathcal{E}$ and $\nicefrac{d\mathcal{E}}{dt}$ are smooth and due to Assumption~\ref{ass:solution} on our solution, the integrand is at least once differentiable in time, and hence bounded by the term $\tau C$ with a solution-dependent $C$, due to the standard a priori error estimates for quadrature formulas.
    Hence, it only remains to estimate the first term.
    We add and subtract the term $ \frac{\nicefrac{d\mcE}{dt}(\theta^{j+1}_\star, \hu^{j+1}_\star) }{\mcE(\theta^{j+1}, \bu^{j+1}) + \bar{C}}$ to obtain
    \begin{align*}
    \abs{
    e^{m+1}_r
    }
    \leq
    \tau
    e^{T/2} r^0
    &
    \left(
    \sum_{j=0}^m
        \left|
        \frac{
        \nicefrac{d\mcE}{dt}(\theta^{j+1}, \bu^{j+1})
        -
        \nicefrac{d\mcE}{dt}(\theta^{j+1}_\star, \hu^{j+1}_\star)
        }{\mcE(\theta^{j+1}, \bu^{j+1}) + \bar{C}}
        \right|
        \right.
        \\
        &
        +
        \left.
        \sum_{j=0}^m 
        \left|
        \frac{d\mcE}{dt}(\theta^{j+1}_\star, \hu^{j+1}_\star)
        \right|
        \left|
        \frac{1}{\mcE(\theta^{j+1}_\star, \hu^{j+1}_\star) + \bar{C}}
        -
        \frac{1}{\mcE(\theta^{j+1}, \hu^{j+1}) + \bar{C}}
        \right|
        +
        C
        \right)
    \end{align*}
    For estimating the first term, we observe for the velocity
    \begin{align*}
        \nu
        \abs{
            \norm{\nabla \bu^{n+1}}^2
            - \norm{\nabla \hu^{n+1}_\star}^2
        }
        &
        \leq 
        \nu
        (
        \norm{\nabla \bu^{n+1}}
        +
        \norm{\nabla \hu^{n+1}_\star}
        )
        \norm{\nabla \be^{n+1}_\hu}
        \leq
        \nu
        C (1 + C_\hu)
        \norm{\nabla \be^{n+1}_\hu}
        \\
        &
        \leq 
        \nu
        C (1 + C_\hu) C_e
        \tau^2
        ,
    \end{align*}
    and similarly for the temperature
    \begin{align*}
        \abs{
            \norm{\nabla \theta^{n+1}}^2
            - \norm{\nabla \theta^{n+1}_\star}^2
        }
        \leq C (1 + C_\theta)
        \norm{\nabla \be^{n+1}_\theta}
        \leq C (1 + C_\theta) C_e
        \tau^2
        .
    \end{align*}
    Further,
    \begin{align*}
        \abs{
            \inner{\hf(\theta^{n+1}), \bu^{n+1}}
            -\inner{\hf(\theta^{n+1}_\star), \bu^{n+1}_\star}
        }
        &
        \leq 
        \abs{
            \inner{\hf(\theta^{n+1}) - \hf(\theta^{n+1}_\star), \bu^{n+1}}
        }
        +
        \abs{
            \inner{\hf(\theta^{n+1}_\star), \be^{n+1}_\hu}
        }
        \\
        &
        \leq 
        \alpha
        \norm{\bu^{n+1}}
        \norm{e^{n+1}_\theta}
        +
        (
        C_{{\bf f}_1}
        +
        \alpha 
        \norm{\theta^{n+1}_\star}
        )
        \norm{\be^{n+1}_\hu}
        \\
        &
        \leq 
        C (C_\hu + 1) C_e \tau^2
        ,
    \end{align*}
    and similarly,
    \begin{align*}
        \alpha^2
        \abs{
            \inner{g(t^{n+1}), \theta^{n+1}}
            -
            \inner{g(t^{n+1}), \theta^{n+1}_\star}
        }
        &
        \leq 
        \alpha^2
        C_g
        \norm{e_\theta^{n+1}}
        \leq
        C
        \tau^2
        .
    \end{align*}
    Combining the last four estimates and using Assumption~\ref{ass:constants} give
    \begin{align*}
        \frac{1}{
            \mathcal{E}(\bu^{n+1}, \theta^{n+1}) + \bar{C}
        }
        \cdot
        \abs{
            \frac{d\mathcal{E}}{dt}( \bu^{n+1}, \theta^{n+1}) -\frac{d\mathcal{E}}{dt}(\hu^{n+1}_\star, \theta^{n+1}_\star)
        }
        \leq
                        C (1+C_e) \tau^2
        .
    \end{align*}
    For the second right-hand side term, we have
    \begin{align*}
        \left|
        \frac{1}{\mcE(\theta^{j+1}_\star, \hu^{j+1}_\star) + \bar{C}}
        -
        \frac{1}{\mcE(\theta^{j+1}, \hu^{j+1}) + \bar{C}}
        \right|
        \leq
        C
        \left|
        \mcE(\theta^{j+1}, \hu^{j+1})
        -
        \mcE(\theta^{j+1}_\star, \hu^{j+1}_\star)
        \right|
        .
    \end{align*}
    We can estimate the denominator by $\bar{C}^{-2} \leq 1$. For the numerator, we combine the estimates 
    \begin{align*}
    \nicefrac{1}{2}
        \abs{
            \norm{\bu^{n+1}}^2
            - \norm{\hu^{n+1}_\star}^2
        }
                \leq 
        C (1 + C_\hu) C_e
        \tau^2
        ,
        \textmd{ and }
            \nicefrac{\bar{\alpha}^2}{2}
        \abs{
            \norm{\theta^{n+1}}^2
            - \norm{\theta^{n+1}_\star}^2
        }
                \leq 
        C (1 + C_\theta) C_e
        \tau^2
        ,
    \end{align*}
    with
    $
        \frac{d\mathcal{E}}{dt}
        (\hu^{n+1}_\star, \theta^{n+1}_\star)
        \leq C
    $, and obtain after simplifying constants
    \begin{align*}
    |e^{m+1}_r|
    \leq \tau C ( 1 + \tau (1 + C_e)  )
    \leq \tau C ( 1 + \tau \sqrt{1+C_0^2})
    .
    \end{align*}
\end{proof}
Finally, we show that all the assumptions on $\eta^n$ and $\xi^n$ are automatically satisfied for $\tau$ small enough.
Hence, the final error estimate of Theorem~\ref{thm:main-error-result} will turn out to be a simple Corollary.
\begin{lemma}[Stability of $\eta$ and $\xi$]\label{lem:boussinesq_etastability}
    Under Assumptions~\ref{ass:domain}, \ref{ass:constants}, \ref{ass:solution}, and \ref{ass:bootstrapping}, we can find a constant $\tau_\star > 0$, such that for all $\tau\leq\tau_\star$ we have $\tau$-independent bounds $c_\eta, c_\xi \geq 0$ such that
    \[
        c_\eta \leq \eta^n
        \quad\textmd{and}\quad
        c_\xi \leq \xi^n,
        \quad
        \textmd{ for all } n \leq N 
        .
    \]
    In addition, we can find a $\tau$-independent $C_0 > 0$, such that
    \[
        \abs{1-\eta^n} \leq C_0 \tau^2,
        \quad
        \textmd{ for all } n \leq N 
        ,
    \]
    where $C_0 = C_0 (T, C_\hu, C_\theta, \tilde{C}_e)$.
\end{lemma}
\begin{proof}
        As in the original derivation~\cite{huang2023stability} for the proof of Navier--Stokes, we proceed with a proof by induction, for which we assume
    \[
        \abs{1 - \xi^n} \leq \sqrt{C_0} \tau, \quad\textmd{holds for all}\quad n \leq m-1
    \]
    and show the hypothesis that
    \[
        \abs{1 - \xi^{m}} \leq \sqrt{C_0} \tau
    \]
    for some constant $C_0 > 1$.
    Due to our initial conditions, we have
    \[ \abs{1 - \xi^0} = 0 \leq \sqrt{C_0} \tau \]
    and can immediately start the induction process.

    We choose our time step size
    \[
        \tau \leq \min\left\{ \frac{1}{2C_0}, 1\right\},
    \]
    and obtain for 
    \[
    \xi^n = 1 + (\xi^n-1) \leq 1 + \abs{1-\xi^n} \leq 1 + \sqrt{C_0} \tau
    \quad
    \textmd{ and }
    \quad
    \xi^n \geq 1 - \abs{1-\xi^n} \geq 1 - \sqrt{C_0} \tau
    \]
    the bounds from the induction hypothesis
    \[
        c_\xi :=  \frac{1}{2} \leq 1 -  \frac{1}{2\sqrt{C_0}} < \abs{\xi^n} < 1 +  \frac{1}{2\sqrt{C_0}} \leq \frac{3}{2} =: C_\xi,
        \quad \forall n \leq m-1,
    \]
    where the first and last inequality are due to $C_0 > 1$.
    Similarly for 
    \[\abs{1 - \eta^n} = \abs{1 - (1 - (1 - \xi^n)^2)}= \abs{1 - \xi^n}^2 \leq C_0 \tau^2 \leq \tau / 2\] we obtain
    \begin{align*}
        c_\eta :=
        \frac{1}{2}
        \leq
        1 - \frac{\tau}{2}
        \leq
        \abs{\eta^n}
        \leq
        1 + \frac{\tau}{2}
        \leq
        \frac{3}{2}
        =: C_\eta
        ,
        \quad \forall n \leq m-1,
    \end{align*}
    where we used the bound on $\tau$.
    Note that $c_\xi$ and $c_\eta$ are all independent of $C_0$. We therefore fulfill the assumptions of Lemma~\ref{lem:stability-heat}, Lemma~\ref{lem:stability-stokes}, Lemma~\ref{lem:error-boussinesq} and Lemma~\ref{lem:error-aux} and get constants $C_\hu$, $C_\theta$, and $\tilde{C}_e$ independent of $C_0$.
    It follows that
    \begin{align*}
        \abs{1-\xi^m}
        &
        =    
        \abs{1-\frac{r^m}{\mathcal{E}(\bu^{m}, \theta^{m}) + \bar{C}}}
        \\
        &
        \leq 
        \abs{\frac{r^m_\star}{\mathcal{E}(\hu^{m}_\star, \theta^{m}_\star) + \bar{C}}-\frac{r^m}{\mathcal{E}(\hu^{m}_\star, \theta^{m}_\star) + \bar{C}}}
        +
        \abs{\frac{r^m}{\mathcal{E}(\hu^{m}_\star, \theta^{m}_\star) + \bar{C}}-\frac{r^m}{\mathcal{E}(\bu^{m}, \theta^{m}) + \bar{C}}}
        \\
        &
        \leq 
        \abs{e_r^m}
        +
        M \abs{\mathcal{E}(\hu^{m}_\star, \theta^{m}_\star) - \mathcal{E}(\bu^{m}, \bar{\theta}^m)}
        ,
    \end{align*}
    where we used $\mathcal{E}(u,\theta) +\bar{C} \geq 1$, and the boundedness of $r^{m}$ due to the weak stability result (Lemma~\ref{lem:weak-stability}).
    \begin{align*}
        \abs{
            \frac{1}{2} \norm{ \bu^{m} }^2
            -\frac{1}{2} \norm{ \hu^{m}_\star }^2
        }
        \leq C (1 + C_\hu) \norm{\be_\hu^m}
        \leq C (1 + C_\hu) \tilde{C}_e \sqrt{1 + C_0^2} \tau^2,
    \end{align*}
    and similarly
    \begin{align*}
        \abs{
            \frac{1}{2} \norm{ \theta^{m} }^2
            -\frac{1}{2} \norm{ \theta^{m}_\star }^2
        }
        \leq C (1 + C_\theta) \norm{e_\theta^m}
        \leq C (1 + C_\theta) \tilde{C}_e \sqrt{1 + C_0^2} \tau^2,
    \end{align*}
    we have
    \begin{align*}
        \abs{1-\xi^m}
        &
        \leq 
        \tau (C_r + \tau M C(1 + C_\hu + C_\theta) \tilde{C}_e \sqrt{1 + C_0^2} )
                        \leq 
        \tau \hat{C} \left(1 + \tau\sqrt{1 + C_0^2} \right)
        .
    \end{align*}
    Note that by construction, $\hat{C}$ is still independent of the choice of $C_0$,
    and without loss of generality, we can assume $\hat{C} > 1$ by enlarging it when necessary.
    We now set $C_0 := (2 \hat{C})^2$ and assume that $\tau \leq \frac{1}{\sqrt{1+C_0^2}} = \frac{1}{\sqrt{1+16\hat{C}^4}}$.

    Thus, for $\tau_\star = \min\{1/(4\hat{C})^2, (1+(2\hat{C})^4)^{-1/2}, 1\}$ we get by induction $\abs{1-\xi^m} \leq 2 \hat{C} \tau = \sqrt{C_0} \tau$ for all $m \leq N$ and have $c_\eta = :1/2 \leq  {\eta^n} \leq C_\eta =: 3/2$ and $C_0 := 2 \hat{C}$ by the previous discussion.
\end{proof}
\section{Numerical results}\label{sec:numerical}
In this subsection, numerical examples illustrate the order of convergence in time, the capability of the algorithm to handle turbulences the scaling of the implementation, and the limitations of our method. All the examples are implemented both in FEniCS~\cite{logg2012automated} to allow an easy modification by the reader, as well as in the high-performance library HyTeG~\cite{kohl2019hyteg}, to enable efficient 3D simulations on fine resolutions.

As a spatial discretization, we use, if not stated otherwise, for the velocity and pressure the $P^2$-$P^1$ Taylor--Hood pairing consisting of continuous piecewise quadratic and piecewise linear polynomials.
For the temperature, we use $H^1$-conforming quadratic finite elements $P^2$.

\subsection{2D Benchmarks}

\paragraph{Convergence rates}
\begin{table}[htb!]
    \centering
    \begin{tabular}{c|cc|cc|cc|cc}
    & \multicolumn{2}{c|}{$\bu$}
    & \multicolumn{2}{c|}{${\hu}$}
    & \multicolumn{2}{c|}{${p}$}
    & \multicolumn{2}{c}{${\theta}$}
    \\
    \hline
    i & 
    error  & rate
    & error  & rate
    & error  & rate
    & error  & rate
    \\
    \hline
    \multicolumn{9}{l}{$P^2$-$P^1$}
    \\
    \hline
    0&      \num[round-mode=figures,round-precision=2]{4.225370e-02}&       -&      \num[round-mode=figures,round-precision=2]{1.123110e-01}&       -&      \num[round-mode=figures,round-precision=2]{7.974690e-02}&  -&      \num[round-mode=figures,round-precision=2]{3.385850e-03}&       -\\ 
    1&      \num[round-mode=figures,round-precision=2]{1.132130e-02}&       \num[round-mode=figures,round-precision=2]{3.73222}&    \num[round-mode=figures,round-precision=2]{2.047520e-02}&       \num[round-mode=figures,round-precision=2]{5.4852}&        \num[round-mode=figures,round-precision=2]{2.099740e-02}&       \num[round-mode=figures,round-precision=2]{3.79794}&    \num[round-mode=figures,round-precision=2]{6.682540e-04}&  \num[round-mode=figures,round-precision=2]{5.06671}\\ 
    2&      \num[round-mode=figures,round-precision=2]{2.885370e-03}&       \num[round-mode=figures,round-precision=2]{3.92369}&    \num[round-mode=figures,round-precision=2]{4.536860e-03}&       \num[round-mode=figures,round-precision=2]{4.51309}&       \num[round-mode=figures,round-precision=2]{5.211920e-03}&       \num[round-mode=figures,round-precision=2]{4.02874}&    \num[round-mode=figures,round-precision=2]{1.586000e-04}&  \num[round-mode=figures,round-precision=2]{4.21345}\\ 
    3&      \num[round-mode=figures,round-precision=2]{7.249760e-04}&       \num[round-mode=figures,round-precision=2]{3.97995}&    \num[round-mode=figures,round-precision=2]{1.074750e-03}&       \num[round-mode=figures,round-precision=2]{4.2213}&        \num[round-mode=figures,round-precision=2]{1.299390e-03}&       \num[round-mode=figures,round-precision=2]{4.01104}&    \num[round-mode=figures,round-precision=2]{3.878350e-05}&  \num[round-mode=figures,round-precision=2]{4.08938}\\ 
    4&      \num[round-mode=figures,round-precision=2]{1.814690e-04}&       \num[round-mode=figures,round-precision=2]{3.99504}&    \num[round-mode=figures,round-precision=2]{2.619140e-04}&       \num[round-mode=figures,round-precision=2]{4.10346}&       \num[round-mode=figures,round-precision=2]{3.251540e-04}&       \num[round-mode=figures,round-precision=2]{3.99623}&    \num[round-mode=figures,round-precision=2]{9.594760e-06}&  \num[round-mode=figures,round-precision=2]{4.04215}\\ 
    5&      \num[round-mode=figures,round-precision=2]{4.538050e-05}&       \num[round-mode=figures,round-precision=2]{3.99884}&    \num[round-mode=figures,round-precision=2]{6.466910e-05}&       \num[round-mode=figures,round-precision=2]{4.05006}&       \num[round-mode=figures,round-precision=2]{8.597810e-05}&       \num[round-mode=figures,round-precision=2]{3.78183}&    \num[round-mode=figures,round-precision=2]{2.386480e-06}&  \num[round-mode=figures,round-precision=2]{4.02047}
    \\
    \hline
    \multicolumn{9}{l}{$P^2$-$P^2$} \\
    \hline
    0&      \num[round-mode=figures,round-precision=2]{4.225370e-02}&       -&      \num[round-mode=figures,round-precision=2]{1.123110e-01}&       -&      \num[round-mode=figures,round-precision=2]{7.975200e-02}&  -&      \num[round-mode=figures,round-precision=2]{3.385850e-03}&       -\\ 
    1&      \num[round-mode=figures,round-precision=2]{1.132130e-02}&       \num[round-mode=figures,round-precision=2]{3.73223}&    \num[round-mode=figures,round-precision=2]{2.047520e-02}&       \num[round-mode=figures,round-precision=2]{5.4852}&        \num[round-mode=figures,round-precision=2]{2.099940e-02}&       \num[round-mode=figures,round-precision=2]{3.79782}&    \num[round-mode=figures,round-precision=2]{6.682540e-04}&  \num[round-mode=figures,round-precision=2]{5.06672}\\ 
    2&      \num[round-mode=figures,round-precision=2]{2.885370e-03}&       \num[round-mode=figures,round-precision=2]{3.9237}&     \num[round-mode=figures,round-precision=2]{4.536850e-03}&       \num[round-mode=figures,round-precision=2]{4.51309}&       \num[round-mode=figures,round-precision=2]{5.213760e-03}&       \num[round-mode=figures,round-precision=2]{4.02769}&    \num[round-mode=figures,round-precision=2]{1.586000e-04}&  \num[round-mode=figures,round-precision=2]{4.21345}
    \end{tabular}
    \caption{
    Errors and rates with the norm in $L^2\big((0,T); L^2(\Omega)\big)$ on a sufficiently refined grid ($256\times 256$), with time step widths of $\tau = T / 2^{i+4}$. {\bf Top:} $P^2$-$P^1$ Taylor--Hood pairing for the velocity and pressure. {\bf Bottom:} $P^2$-$P^2$ pairing for velocity and pressure.
    }
    \label{tab:convergence}
\end{table}
To verify the proven convergence rates, 
we consider the Boussinesq equation \eqref{eq:gen:heat} and \eqref{eq:gen:momentum} in the region $\Omega=(0,1)^2$ and the end time $T=\pi$ with the manufactured solution
\begin{align*}
    \hu(t,x,y) & =
    \sin(t)
    \,
    \left(
    \sin^2(2\pi x)\sin(2\pi y)\cos(2\pi y),\;
    \sin(2\pi x)\cos(2\pi x)\sin^2(2\pi y)
    \right)^T, 
    \\
    p(t,x,y)   & =\sin(t)\sin(2\pi x)\sin(2\pi y), \quad\textmd{and}     \\
    \theta(t,x,y)   & =\sin(t)\sin(2\pi x)\sin(2\pi y).
\end{align*}
Note that $\nabla\cdot \hu=0$, and that $u$ and $\theta$ are zero on $\partial\Omega$. 
We choose $\hf_2(\theta)=(\theta, 0)^T$, and $\hf_1$ and $g$ such that we obtain the manufactured solution.

In Table~\ref{tab:convergence}, the error for $\bu$, $\hu$, $p$ and $\theta$ for the $L^2\big((0,T);L^2(\Omega)\big)$-norm, and the resulting convergence rates are given for a time discretization with $l=1$ and $k=3$.
All errors approach the theoretically proven second-order convergence rate, $\hu$ and $\theta$ from above, while $\bu$ and $p$ approach it from below.
As expected, the error for the unscaled field $\bu$ is slightly better than for the by $\eta$ rescaled velocity field $\hu$.
The slight drop in the rate of $p$ is due to the spatial discretization since, at this point, the spatial error starts to dominate the pressure $p$, which is only approximated by continuous piecewise linear polynomials. 
Changing this space to $P^2$ yields similar convergence rates as before. This suggests that, in practice, inf-sup stability is not necessary for the scheme to work, however, the scheme does not benefit from a higher best approximation order in $p$.

\begin{figure}[htb]
    \centering
    \includegraphics[width=\textwidth]{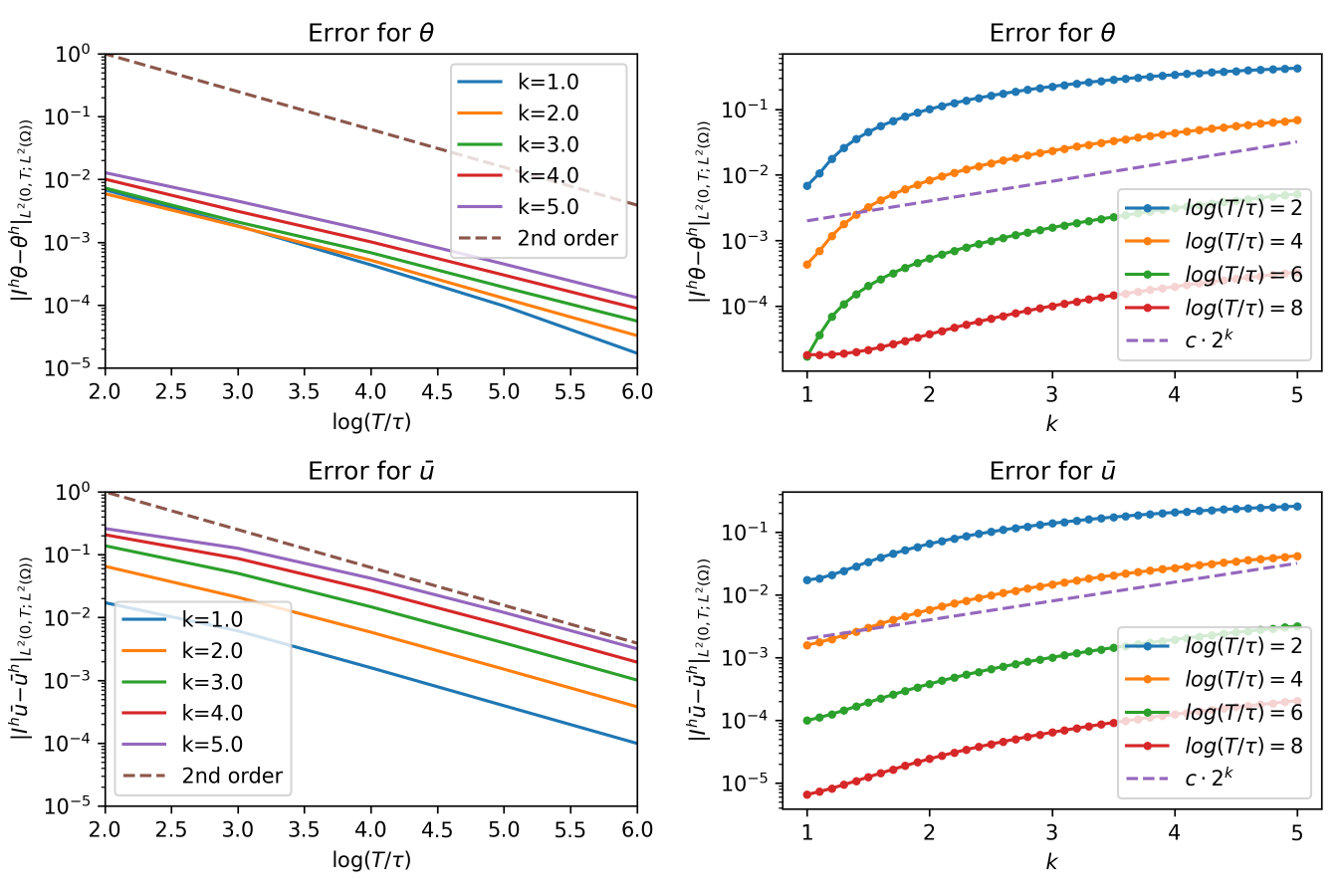}
    \caption{
    The $L^2\big((0,T);L^2(\Omega)\big)$-error for real-valued $k$ and different time step sizes, and a spatial grid of $256\times 256$.
    }
    \label{fig:combi-pic}
\end{figure}
In Figure~\ref{fig:combi-pic}, we present a more detailed convergence study using different values for $k$ and time step widths.
On the left, we see that a small value in $k$ consistently leads to smaller errors for all fields $\bu$, and $\theta$ with a difference in the error up to one order of magnitude.
Hence, we should choose $k$ as small as possible from a theoretical point of view to obtain the smallest errors, i.e., setting $k=3$ and $l=1$.
On the right, we have a real-valued $k$.
Especially the error of the velocity $\hu$ reacts strongly to small values for $k$, while the effect is less prominent for large $k$.
The temperature $\theta$ is not affected as strongly by smaller by $k$.

\paragraph{2D-Benchmark}
As a less synthetic example, we use a simple Marsigli flow benchmark, which shows that the algorithm can handle turbulence and allows comparisons with other time discretizations~\cite{jiang2023unconditionally}.
For this, consider the Boussinesq equation \eqref{eq:boussinesq:heat}, \eqref{eq:boussinesq:momentum} and \eqref{eq:boussinesq:mass} on $\Omega=(0,8)\times(0,1)$ with $T=10$, $\text{Re}=5000$, $\text{Ri}=4$ and $\text{Pr}=1$. As boundary conditions, we use $\hu|_{\partial\Omega}=\mathbf{0}$ and $\nicefrac{\partial \theta}{\partial \vec{n}}|_{\partial\Omega}=0$, and as initial conditions 
\[
\hu(0,x)=\mathbf{0},\quad p(0,x)=0,\quad \textmd{and} \quad
\theta(0,x)=\begin{cases}\nicefrac{3}{2}&\textmd{for } x_1<4\\1&\text{otherwise}\end{cases},
\quad \textmd{for} \quad x \in \Omega.
\]
Since the velocity and pressure are initially zero, they only change over time through the right-hand side of the velocity equation. The variation in the initial temperature leads to a larger force upwards on the left side of the domain, pushing the higher temperature to the right side of the domain. Mass balance leads to the lower temperature being pushed onto the lower part of the domain to the left side of the domain. The high Reynolds number creates turbulences. The result with a spatial discretization of $h=1/64$ and a temporal discretization of $\tau=T/2^{17}$ can be seen in Figure~\ref{fig:marsigli} at the time points $t=2$ and $8$. Streamlines to the corresponding time points are visualized in the lower part of this figure.

\begin{figure}
    \centering
    \includegraphics[width=\textwidth]{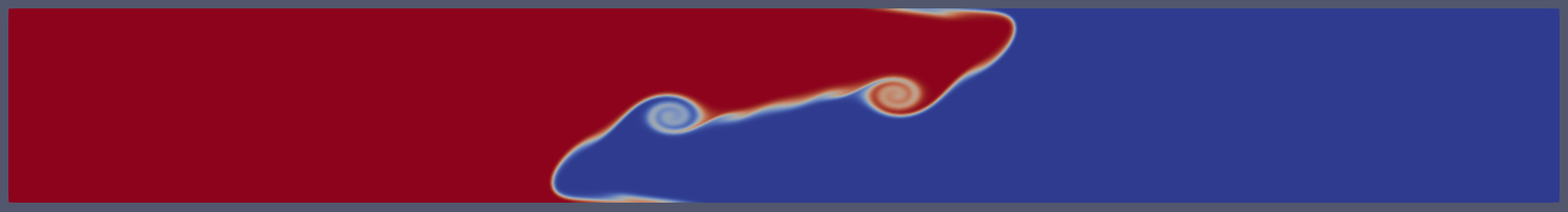}\\
            \includegraphics[width=\textwidth]{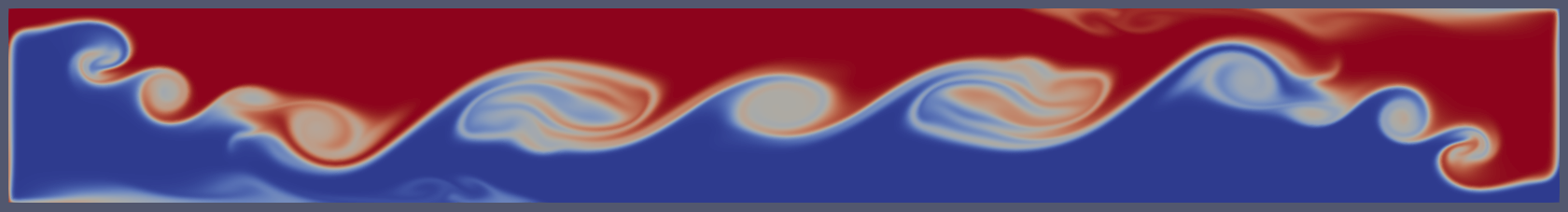}\\
    \includegraphics[width=\textwidth]{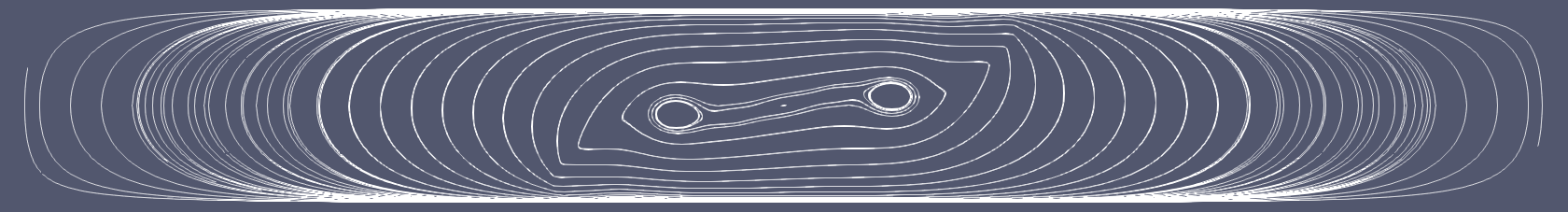}\\
            \includegraphics[width=\textwidth]{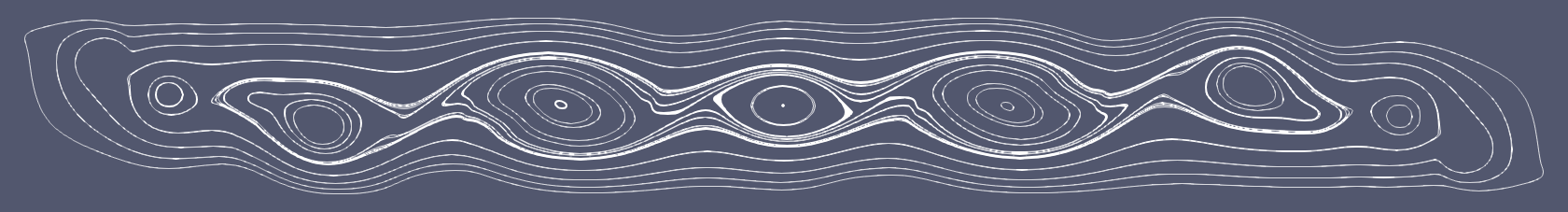}
    \caption{The temperature (top) and the velocity (bottom) for the Marsigli flow at $t=2,8$ with $h=\frac{1}{64}$ and $\tau=\frac{T}{2^{17}}$ illustrated via the color scale red $=1.5$ and blue $=1.0$ and particle streamlines.}
    \label{fig:marsigli}
\end{figure}

\paragraph{Stabilization in space}
While we have shown stability for a semi-discretization in time, which is spatially continuous, we cannot expect stability for a completely explicit advection term in a discretized space.
Hence, we introduce the following spatial stabilization terms mimicking the first derivative in time, scaled with $h$, to retain the convergence order in time.
As possible candidates we test the additional terms $S_a = 2 \tau c_s h \cdot (-\Delta) D^k(\bu^{n+1}, {\hu}^n, {\hu}^{n-1})$ and $S_b =2 \tau c_s h \cdot (-\Delta) (\bu^{n+1} - \hu^{n-1})$, with local mesh width $h$ and problem dependent scaling parameter $c_s = \nicefrac{1}{2}$.
As a benchmark problem, we use the second example in~\cite{huang2023stability}, i.e., $\Omega = (-1,1)^2$ with initial condition 
\[
\hu(0,x) = \begin{pmatrix}(1+x_1)(1-x_2) \tanh (\rho( \alpha(x_2) x_2+0.5)) \\ \delta \sin(\pi x_1)\end{pmatrix}, \textmd{ where } \alpha(x_2) = \begin{cases}
    +1, & x_2 \leq 0 \\ -1 & x_2 > 0
\end{cases}
,
\]
for the velocity where we set $\rho=100$, and  $\delta=0.5$.
The initial pressure $p(0,x)=0$ and temperature $\theta(0, x) = 0$ are set to zero.
As physical problem parameters, we set $\mathbf{f} = 0$, $\nu = 0.005$, and deactivate the temperature with $\mathbf{g} = 0$.
\begin{figure}[htb]
    \centering
    \includegraphics[width=0.85\linewidth]{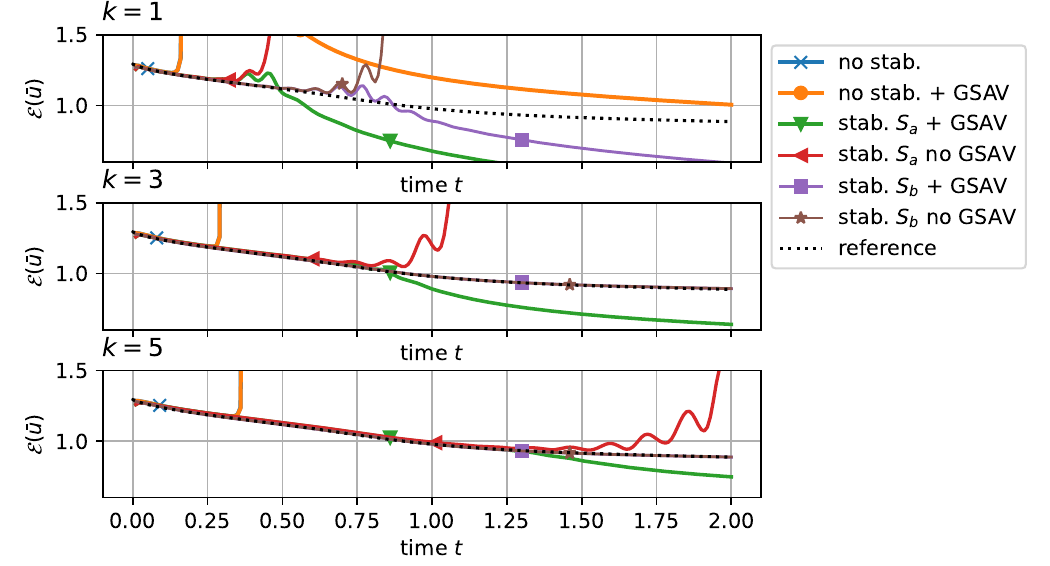}
    \caption{Energy over time for different $k$ and $\tau=10^{-2}$.}
    \label{fig:gsav-spatial-stabilization}
\end{figure}
The energy of the resulting simulations for $k=1,3$ and $5$ on a $64\times64$ grid is plotted against time in Figure~\ref{fig:gsav-spatial-stabilization} for a much larger time step width of $\tau=10^{-2}$, compared to the $\tau = \num{2e-3}$ in~\cite{huang2023stability}.
As expected, increasing $k$ always leads to more stable schemes independent of the stabilizations used.
Using no stabilization at all leads to schemes where the energy $\mathcal{E}$ explodes, which happens earlier when $k$ is smaller.
Adding the GSAV stabilization in time limits the growth of $\bu$ by deactivating the advective term.
Using the stabilization term $S_a$ without and with GSAV leads to either an exploding or strongly damped energy.
The term $S_b$ sufficiently stabilizes the scheme for $k=3$ and $5$, while it is not stable for $ k=1$.

All in all, the GSAV approach alone is insufficient to achieve stability using finite elements. However, classical in-space stabilization techniques can recover stability in the space-time domain.

\subsection{3D Benchmark}
The third example shows that the implemented code scales well and can handle more complex scenarios in 3D. For this, consider the Boussinesq equation \eqref{eq:boussinesq:heat}, \eqref{eq:boussinesq:momentum} and \eqref{eq:boussinesq:mass} with $\text{Re}=5000$, $\text{Ri}=4$ and $\text{Pr}=1$ in the domain $\Omega=(0,1)^3$ and $T=2.5$. As initial condition, we use $\hu=\mathbf{0}$, $p=0$ and
\[
\theta(0,x)=\frac{5}{4}+\frac{1}{4}\cdot\tanh\Big(3-(x_1^2+x_2^2)+2e^{-50\big((x_1-\frac12)^2+(x_2-\frac12)^2\big)}-10x_3\Big).
\]
This is a smooth version of the step function that is $1$ if $z<f(x,y)$ and $\nicefrac{3}{2}$ else, where $\hf$ is a constant at $0.3$ with a quadratic decay in $x$- and $y$-direction as well as a Gaussian bump in the center of $(0,1)^2$. The contour plot,  for $\theta(0,x,y,z)=\nicefrac{5}{4}$, is depicted on the left of Fig.~\ref{fig:dd}.

\begin{figure}[htb]
\centering
\includegraphics[width=\textwidth]{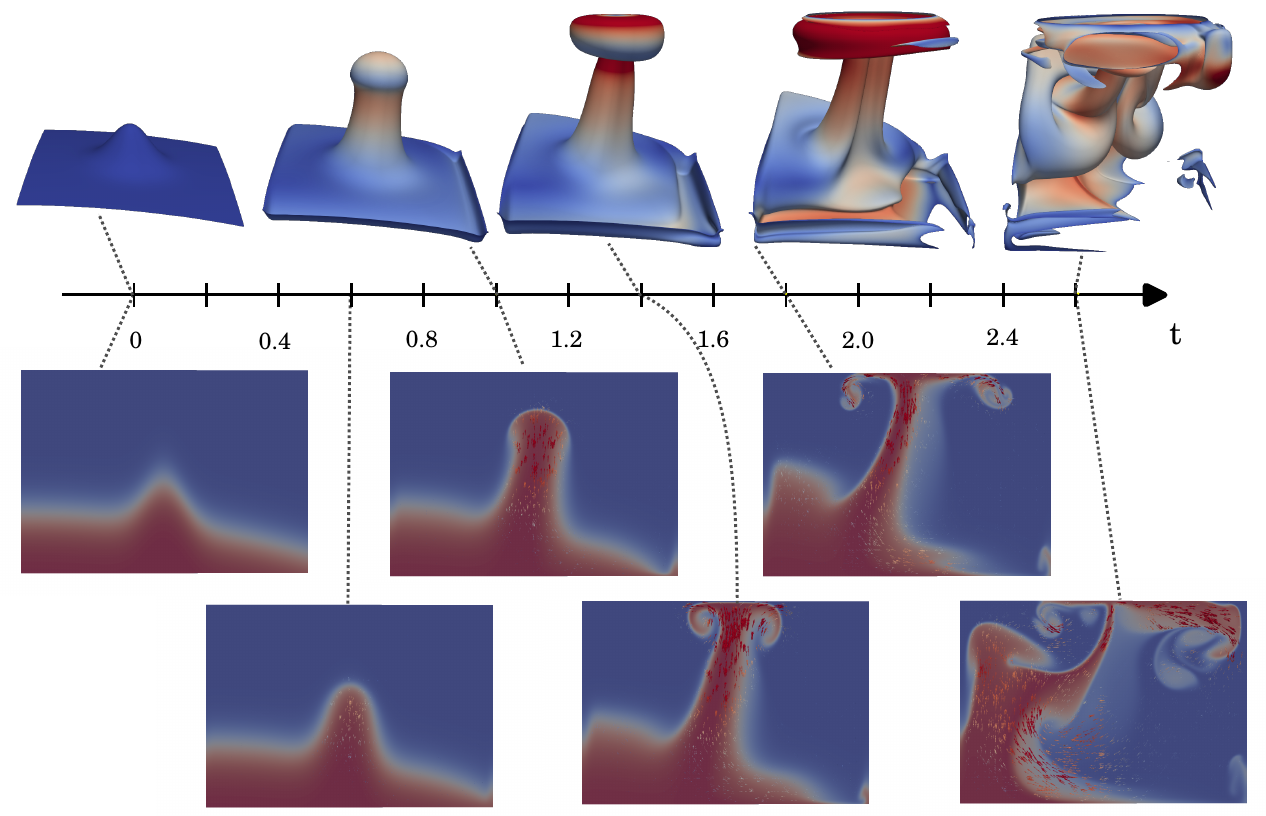}
\caption{
Simulation results for the temperature with the times, going from left to right.
Above the time-axis: contour plots with $\theta = 1.25$.
Below the time-axis: Slice in the $x=y$ plane with $\theta$ as background color and the $\hu$ indicated by small vectors.
}
\label{fig:dd}
\end{figure}
The grid size is given by $\tau=\nicefrac{T}{2^{14}}\approx1.5\cdot10^{-4}$ and $h=2^{-9}$ resulting in approximately $2.2\cdot10^9$ spatial unknowns for each of the $2^{14}$ time steps (consisting of $5.4\cdot10^8$ unknowns for each component of $\hu$ and for $\theta$ as well as $6.8\cdot10^7$ unknowns for $p$). We use a geometric multi-grid solver with Gauss-Seidel as smoother and the CG method as a coarse grid solver. The simulation is run on 256 cores with $3.1$GHz as the maximal CPU frequency. 

The results of this simulation are shown in two different ways: Above the time-axis, the contour plot of the temperature with $\theta=1.25$ is shown with a coloring of the magnitude of the velocity; and below the time-axis, the plane $x=y$ is shown with the temperature (in the range of $1.0$ to $1.5$) with arrows indicating in both, size and color, the velocity $\hu$. 

With the force on the velocity being linear in the temperature in the $z$-direction, the Gaussian bump creates the most upward force, resulting in a plume-like structure forming between $t=0.8$ and $t=1.0$. As soon as the plume hits the upper boundary with imposed no-slip boundary conditions, the plume spreads in the $xy$-direction, resulting in faster (and thinner) twirls ($t=1.6$), which lead to even thinner twirls ($t=2.2$ until $t=2.6$). Meanwhile, the quadratic decay of the initial state in both the $x$- and $y$-direction leads to an imbalance in upward force, resulting in the cooler temperature being pushed down near $(x,y)=(1,1)$. Note that the boundary condition for $\hu$ is a no-slip condition everywhere with the obvious results in the simulation.

\section{Conclusion}\label{sec:outlook}
We have introduced a time-stepping scheme for the Boussinesq equation based on the discretization in~\cite{huang2023stability} for the Stokes equation using BDF(k) and GSAV to stabilize 
the algorithm. 
We have reformulated it to be suitable for a finite element discretization and presented error estimates for the time-stepping scheme showing second-order convergence in time for all BDF($k$) schemes with $k\geq 3$.

Further, we have verified these rates numerically and demonstrated the scheme's applicability to the Marsigli flow in 2D and 3D.
Even though the asymptotic behavior is independent of $k$, small values in $k$ lead to smaller error constants.
Finally, we demonstrated that GSAV alone is not enough for a truly robust algorithm, though spatial stabilizations can improve the algorithm's stability. 
Enlarging $k$ hereby yields additional stability.
Hence, there is a trade-off between choosing $k$ as large as possible for maximum stability and making it as small as possible for a minimal error constant.

\appendix
\section{Appendix}
\subsection{Proofs}\label{sec:app:proofs}
\begin{delayedproof}{lem:time-derivative-and-extrapolation}
    With a Taylor expansion around $t_{n+k}$, we immediately obtain the integral remainder terms
    \begin{align*}
        \zeta_{n+1,k}(v)
        :=                  &
        k \int^{t^{n+1}}_{t^{n+k}} \partial^2_t{v}(s) (t^{n+1}-s)ds
        -
        (k-1) \int^{t^{n}}_{t^{n+k}} \partial^2_t{v}(s) (t^{n}-s)ds
        \\
        \Xi_{n+1}(v)
        :=                  &
        \frac{1}{2}
        \left(
        (2k+1) \int^{t^{n+1}}_{t^{n+k}}
        \partial^3_t v(s)(t^{n+1}-s)^2 ds
                                        - 4k \int^{t^{n}}_{t^{n+k}}
        \partial^3_t v(s)(t^{n}-s)^2 ds
        \right.
        \\ &
        \qquad
        \left.
        +(2k-1) \int^{t^{n-1}}_{t^{n+k}}
        \partial^3_t v(s)(t^{n-1}-s)^2 ds,
        \right)
        \quad\textmd{and}
        \\
        \Lambda_{n+1}(r) := & \int^{t^{n}}_{t^{n+1}} \partial^2_t{r}(s)(t_n - s) ds
        .
    \end{align*}
    For the integral estimate, we start with the first term $\zeta_{n+1,k}(v)$ and get
    \begin{align*}
    \zeta_{n+1,k}(v) = 
    \int^{t^{n+1}}_{t^{n+k}}
    \partial^2_{t}{v}(s) (t^{n+k} - s) ds
    -
    (k-1)
    \int^{t^{n}}_{t^{n+1}}
    \partial^2_{t}{v}(s) (t^{n} - s) ds.
    \end{align*}
    Hence, applying the norm, using Cauchy-Schwarz and Young, on both terms yields 
    \begin{align*}
    \norm{ \zeta_{n+1,k}(v) }^2_V
    \leq
    \frac{2}{3}
    \tau^3 (k-1)^3
    \norm{\partial^2_t v}^2_{L^2((t^{n+1}, t^{n+k}), V)}
    +
    \frac{2}{3}
    (k-1)^2
    \tau^3
    \norm{\partial^2_t v}^2_{L^2((t^{n}, t^{n+1}), V)}
    .
    \end{align*}
    Further, summing up over n and using that the integrals of the first term overlap at most $k-1$ times yields
    \begin{align*}
    \sum_{n=0}^{m-1}
    \norm{ \zeta_{n+1,k}(v) }^2_V
    \leq
    \frac{2}{3}
    \tau^3 
    (k-1)^2 
    \max( (k-1) k , 2 )
    \norm{\partial^2_t v}^2_{L^2((0,T^*), V)}
    ,
    \end{align*}
    where we further used $ ((k-1)^2 + 1) \leq (k-1)k$ for $k \geq 2$.
    The proof for $\Xi_{n+1}(v)$  works analogously.
We note that smaller upper bounds can be obtained by not independently bounding the terms. However, we omit this as the constants are of minor importance.
\end{delayedproof}

\begin{delayedproof}{lem:abb-equation-2}
    We first show the equality
    \begin{align*}
        (\tD^k v^{n+1}, \delta^{k+1} v^{n+1})_V
        = &
        a(\norm{v^{n+1}}^2_V - \norm{v^n}^2_V)
        +
        (\norm{b v^{n+1} - c v^n}^2_V
        -\norm{b v^n - c v^{n-1}}^2_V)
        +
        \nonumber
        \\
          & \quad
        + d \norm{v^{n+1} - 2 v^n + v^{n-1}}^2_V
        + 2 \norm{v^{n+1} - v^n}^2_V,
    \end{align*}
    where the constants $a, b, c$ and $d$ are given by
    \begin{align*}
        a = \frac{3}{2 (k+1)}, \quad
        c = \sqrt{k^{2} + \frac{k}{2} - \frac{1}{2}}, \quad
                b = c^{-1}\left(k^{2} + \frac{3 k}{2} - 1\right),
        \quad\textmd{and}\quad
        d = c^2
        .
    \end{align*}
    Expanding and collecting the terms of the left-hand side, we get
    \begin{align*}
        4 k^{2} \norm{ v^{n}}^2_V
        + \left(- 2 k^{2} + k\right) \inner{ v^{n-1},  v^{n}}_V
        + \left(2 k^{2} + k - 1\right) \inner{ v^{n-1},  v^{n+1}}_V
        + \\
        + \left(- 6 k^{2} - 5 k\right) \inner{ v^{n},  v^{n+1}}_V
        + \left(2 k^{2} + 3 k + 1\right) \norm{ v^{n+1}}^{2}_V
        ,
    \end{align*}
    while the right-hand-side becomes
    \begin{align*}
        \left(- c^{2} + d\right) \norm{ v^{n-1}}^{2}_V
        + \left(- a - b^{2} + c^{2} + 4 d + 2\right) \norm{ v^{n}}^{2}_V
        + \left(2 b c - 4 d\right) \inner{ v^{n-1},  v^{n}}_V
        + \\
        + 2 d \inner{ v^{n-1},  v^{n+1}}_V
        + \left(- 2 b c - 4 d - 4\right) \inner{ v^{n},  v^{n+1}}_V
        + \left(a + b^{2} + d + 2\right) \norm{ v^{n+1}}^{2}_V
        .
    \end{align*}
    By comparing the individual terms, a straightforward matching shows
    \begin{align*}
        \inner{ \delta^k v^{n+1}, \delta^{k+1} v^{n+1} }_V
        =
        \hat{a}
        \norm{\delta^{k+1}v^{n+1}}^2_V
        + \hat{b} \norm{v^{n+1} + v^{n}}^{2}_V
        +  \left(\hat{d} + \hat{f}\right) \norm{v^{n+1}}^{2}_V
        - \hat{d} \norm{v^{n}}^{2}_V
        ,
    \end{align*}
    where
    \begin{align*}
        \hat{a}= \frac{4 k^{2} - 1 - \epsilon}{4 k^{2} + 4 k + 1},
        \quad
        \hat{b}= 
        \frac{k+1/2+k(1-\epsilon)}{(2k+1)^2},
        \quad
        \hat{d} = 
        \frac{k + \frac{1}{2} - \epsilon k}{2k+1},
        \quad
        \hat{f} = \epsilon.
    \end{align*}
    Expanding the left-hand side terms yields
    \begin{align*}
        \inner{\delta^k v^{n+1}, \delta^{k+1} v^{n+1}}_V
        =
        \norm{v^{n+1}}^{2}_V k \left(k + 1\right) + \inner{v^{n+1}, v^{n}}_V \left(1 - 2 k^{2}\right) + \norm{v^{n}}^{2}_V k \left(k - 1\right)
        ,
    \end{align*}
    while the right-hand side gives
    \begin{align*}
        \norm{v^{n+1}}^{2}_V \left(\hat{a} k^{2} + 2 \hat{a} k + \hat{a} + \hat{b} + \hat{d} + \hat{f}\right) + \inner{v^{n+1}, v^{n}}_V \left(- 2 \hat{a} k^{2} - 2 \hat{a} k + 2 \hat{b}\right) + \norm{v^{n}}^{2}_V \left(\hat{a} k^{2} + \hat{b} - \hat{d}\right)
        .
    \end{align*}
    Again, the result follows from a one-to-one comparison of the coefficients.
\end{delayedproof}

\begin{acknowledgements}
    The work of the second and third authors was supported by the Federal Ministry of
    Education and Research (BMBF) as part of the “Multi-physics
    simulations for Geodynamics on heterogeneous Exascale Systems” (CoMPS)
    project (FKZ 16ME0651) inside the federal research program on “High-Performance and
    Supercomputing for the Digital Age 2021-2024 – Research and
    Investments in High-Performance Computing”.
\end{acknowledgements}

\section*{Declaration}

\paragraph{Competing interests} The authors declare that they have no competing interests.

\bibliographystyle{spmpsci}      \bibliography{literature}   

\begin{thebibliography}{10}
\providecommand{\url}[1]{{#1}}
\providecommand{\urlprefix}{URL }
\expandafter\ifx\csname urlstyle\endcsname\relax
  \providecommand{\doi}[1]{DOI~\discretionary{}{}{}#1}\else
  \providecommand{\doi}{DOI~\discretionary{}{}{}\begingroup
  \urlstyle{rm}\Url}\fi

\bibitem{BAKER1994261}
Baker, A., Williams, P., Kelso, R.: Development of a robust finite element
  {CFD} procedure for predicting indoor room air motion.
\newblock Building and Environment \textbf{29}(3), 261--273 (1994).
\newblock \doi{10.1016/0360-1323(94)90022-1}.
\newblock Special Issue Papers from Indoor Air '93

\bibitem{castillo2024explicit}
Castillo, D.R., Kaya, U., Richter, T.: An explicit time integration method for
  {B}oussinesq approximation  (2024)

\bibitem{chen2023unconditional}
Chen, C., Zhang, T.: Unconditional stability of first and second orders
  implicit/explicit schemes for the natural convection equations.
\newblock Computers \& Mathematics with Applications \textbf{139}, 152--172
  (2023)

\bibitem{cheng2018multiple}
Cheng, Q., Shen, J.: Multiple scalar auxiliary variable ({MSAV}) approach and
  its application to the phase-field vesicle membrane model.
\newblock SIAM Journal on Scientific Computing \textbf{40}(6), A3982--A4006
  (2018)

\bibitem{damanik2009monolithic}
Damanik, H., Hron, J., Ouazzi, A., Turek, S.: A monolithic fem-multigrid solver
  for non-isothermal incompressible flow on general meshes.
\newblock Journal of Computational Physics \textbf{228}(10), 3869--3881 (2009)

\bibitem{davis2002operator}
Davis, D., B{\"a}nsch, E.: An operator-splitting finite-element approach to the
  8: 1 thermal-cavity problem.
\newblock International journal for numerical methods in fluids \textbf{40}(8),
  1019--1030 (2002)

\bibitem{ding2024optimal}
Ding, Q., Hou, Y., He, X.: Optimal error estimates of a second-order fully
  decoupled finite element method for the nonstationary generalized
  {B}oussinesq model.
\newblock Journal of Computational and Applied Mathematics \textbf{450}, 116001
  (2024)

\bibitem{gassmoller2020formulations}
Gassm{\"o}ller, R., Dannberg, J., Bangerth, W., Heister, T., Myhill, R.: On
  formulations of compressible mantle convection.
\newblock Geophysical Journal International \textbf{221}(2), 1264--1280 (2020)

\bibitem{guermond2003new}
Guermond, J., Shen, J.: A new class of truly consistent splitting schemes for
  incompressible flows.
\newblock Journal of computational physics \textbf{192}(1), 262--276 (2003)

\bibitem{guermond2006overview}
Guermond, J.L., Minev, P., Shen, J.: An overview of projection methods for
  incompressible flows.
\newblock Computer methods in applied mechanics and engineering
  \textbf{195}(44-47), 6011--6045 (2006)

\bibitem{hawkins2024analysis}
Hawkins, E.: Analysis of the {P}icard-{N}ewton iteration for the incompressible
  {B}oussinesq equations.
\newblock arXiv preprint arXiv:2408.16872  (2024)

\bibitem{hou2022decoupled}
Hou, Y., Yan, W., Li, M., He, X.: A decoupled and iterative finite element
  method for generalized {B}oussinesq equations.
\newblock Computers \& Mathematics with Applications \textbf{115}, 14--25
  (2022)

\bibitem{huang2022new}
Huang, F., Shen, J.: A new class of implicit--explicit {BDF-k} {SAV} schemes
  for general dissipative systems and their error analysis.
\newblock Computer Methods in Applied Mechanics and Engineering \textbf{392},
  114718 (2022)

\bibitem{huang2023stability}
Huang, F., Shen, J.: Stability and error analysis of a second-order consistent
  splitting scheme for the {N}avier--{S}tokes equations.
\newblock SIAM Journal on Numerical Analysis \textbf{61}(5), 2408--2433 (2023)

\bibitem{jiang2019pressure}
Jiang, N.: A pressure-correction ensemble scheme for computing evolutionary
  {B}oussinesq equations.
\newblock Journal of Scientific Computing \textbf{80}, 315--350 (2019)

\bibitem{jiang2023unconditionally}
Jiang, N., Yang, H.: Unconditionally stable, second order, decoupled ensemble
  schemes for computing evolutionary {B}oussinesq equations.
\newblock Applied Numerical Mathematics \textbf{192}, 241--260 (2023)

\bibitem{kaya2024error}
Kaya, U., Richter, T.: Error analysis of a pressure correction method with
  explicit time stepping.
\newblock arXiv preprint arXiv:2407.11159  (2024)

\bibitem{kohl2019hyteg}
Kohl, N., Th{\"o}nnes, D., Drzisga, D., Bartuschat, D., R{\"u}de, U.: The hyteg
  finite-element software framework for scalable multigrid solvers.
\newblock International Journal of Parallel, Emergent and Distributed Systems
  \textbf{34}(5), 477--496 (2019)

\bibitem{li2023error}
Li, X., Shen, J.: Error estimate of a consistent splitting {GSAV} scheme for
  the {N}avier-{S}tokes equations.
\newblock Applied Numerical Mathematics \textbf{188}, 62--74 (2023)

\bibitem{liu2007stability}
Liu, J.G., Liu, J., Pego, R.L.: Stability and convergence of efficient
  {N}avier-{S}tokes solvers via a commutator estimate.
\newblock Communications on Pure and Applied Mathematics: A Journal Issued by
  the Courant Institute of Mathematical Sciences \textbf{60}(10), 1443--1487
  (2007)

\bibitem{logg2012automated}
Logg, A., Mardal, K.A., Wells, G.: Automated solution of differential equations
  by the finite element method: The FEniCS book, vol.~84.
\newblock Springer Science \& Business Media (2012)

\bibitem{lube2008stabilized}
Lube, G., Knopp, T., Rapin, G., Gritzki, R., R{\"o}sler, M.: Stabilized finite
  element methods to predict ventilation efficiency and thermal comfort in
  buildings.
\newblock International journal for numerical methods in fluids \textbf{57}(9),
  1269--1290 (2008)

\bibitem{MADALOZZO20145883}
Madalozzo, D., Braun, A., Awruch, A., Morsch, I.: Numerical simulation of
  pollutant dispersion in street canyons: Geometric and thermal effects.
\newblock Applied Mathematical Modelling \textbf{38}(24), 5883--5909 (2014).
\newblock \doi{10.1016/j.apm.2014.04.041}

\bibitem{pan2022monolithic}
Pan, X., Kim, K.H., Choi, J.I.: Monolithic projection-based method with
  staggered time discretization for solving non-{O}berbeck--{B}oussinesq
  natural convection flows.
\newblock Journal of Computational Physics \textbf{463}, 111238 (2022)

\bibitem{paszynski2020massively}
Paszy{\'n}ski, M., Siwik, L., Podsiad{\l}o, K., Minev, P.: A massively parallel
  algorithm for the three-dimensional {N}avier-{S}tokes-{B}oussinesq
  simulations of the atmospheric phenomena.
\newblock In: International Conference on Computational Science, pp. 102--117.
  Springer (2020)

\bibitem{pollock2021acceleration}
Pollock, S., Rebholz, L.G., Xiao, M.: Acceleration of nonlinear solvers for
  natural convection problems.
\newblock Journal of Numerical Mathematics \textbf{29}(4), 323--341 (2021)

\bibitem{shen2018convergence}
Shen, J., Xu, J.: Convergence and error analysis for the scalar auxiliary
  variable ({SAV}) schemes to gradient flows.
\newblock SIAM Journal on Numerical Analysis \textbf{56}(5), 2895--2912 (2018)

\bibitem{shen2018scalar}
Shen, J., Xu, J., Yang, J.: The scalar auxiliary variable ({SAV}) approach for
  gradient flows.
\newblock Journal of Computational Physics \textbf{353}, 407--416 (2018)

\bibitem{shen2007error}
Shen, J., Yang, X.: Error estimates for finite element approximations of
  consistent splitting schemes for incompressible flows.
\newblock Discrete and Continuous Dynamical Systems Series B \textbf{8}(3), 663
  (2007)

\bibitem{spiegel1971convection}
Spiegel, E.A.: Convection in stars {I.} basic {B}oussinesq convection.
\newblock Annual review of astronomy and astrophysics \textbf{9}(1), 323--352
  (1971)

\bibitem{takhirov2021direction}
Takhirov, A., Frolov, R., Minev, P.: A direction splitting scheme for
  {N}avier--{S}tokes--{B}oussinesq system in spherical shell geometries.
\newblock International Journal for Numerical Methods in Fluids
  \textbf{93}(12), 3507--3523 (2021)

\bibitem{de1983natural}
de~Vahl~Davis, G., Jones, I.: Natural convection in a square cavity: a
  comparison exercise.
\newblock International Journal for numerical methods in fluids \textbf{3}(3),
  227--248 (1983)

\bibitem{ward2000application}
Ward, L.W.: Application of the two-dimensional {N}avier-{S}tokes equations to
  steam cooling in asymmetrically heated channels.
\newblock Nuclear Technology \textbf{131}(1), 69--81 (2000).
\newblock \doi{10.13182/NT00-A3105}

\bibitem{yang2023error}
Yang, Y.B., Huang, B.C., Jiang, Y.L.: Error estimates of an operator-splitting
  finite element method for the time-dependent natural convection problem.
\newblock Numerical Methods for Partial Differential Equations \textbf{39}(3),
  2202--2226 (2023)

\bibitem{zeytounian2003joseph}
Zeytounian, R.K.: Joseph {B}oussinesq and his approximation: a contemporary
  view.
\newblock Comptes Rendus Mécanique \textbf{331}(8), 575--586 (2003).
\newblock \doi{10.1016/S1631-0721(03)00120-7}

\bibitem{zhang2024error}
Zhang, J., Yuan, L., Chen, H.: Error estimate of a fully decoupled numerical
  scheme based on the scalar auxiliary variable ({SAV}) method for the
  {B}oussinesq system.
\newblock Communications in Nonlinear Science and Numerical Simulation p.
  108102 (2024)

\bibitem{zhang2018decoupled}
Zhang, T., Jin, J., Jiang, T.: The decoupled
  {C}rank--{N}icolson/{A}dams--{B}ashforth scheme for the {B}oussinesq
  equations with nonsmooth initial data.
\newblock Applied Mathematics and Computation \textbf{337}, 234--266 (2018)

\end{thebibliography}
\end{document}